\documentclass[11pt,a4paper]{amsart}
\pdfoutput=1

\usepackage[top=1in, bottom=1in, left=1in, right=1in]{geometry}
\usepackage{amsmath}
\usepackage{amssymb}
\usepackage{mathrsfs} 
\usepackage{graphicx} 
\usepackage{caption}
\usepackage{enumerate} 
\usepackage{soul} 
\usepackage{color}
\usepackage{multirow} 

\usepackage[colorlinks=false, pdfborder={0 0 0}]{hyperref}

\newtheorem{theorem}{Theorem}
\newtheorem{proposition}[theorem]{Proposition}

\newtheorem{corollary}[theorem]{Corollary}

\newtheorem{observation}{Observation}

\newtheorem*{thm_nonumber}{Theorem}

\newcommand{\N}{\mathbb{N}}

\newcommand{\R}{\mathbb{R}}
\newcommand{\bbS}{\mathbb{S}}
\newcommand{\bbV}{\mathbb{V}}
\newcommand{\X}{\mathbb{X}}
\newcommand{\Z}{\mathbb{Z}}

\newcommand{\cD}{\mathscr{D}}
\newcommand{\cE}{\mathscr{E}}
\newcommand{\cF}{\mathscr{F}_{\lambda}}

\newcommand{\cH}{\mathscr{H}}
\newcommand{\cI}{\mathscr{I}}
\newcommand{\cO}{\mathcal{O}}
\newcommand{\cP}{\mathscr{P}}
\newcommand{\cQ}{\mathscr{Q}}
\newcommand{\cR}{\mathcal{R}}

\newcommand{\cT}{\mathscr{T}}
\newcommand{\cX}{\mathscr{X}}

\newcommand{\bfv}{\mathbf{v}}

\newcommand{\bfw}{\mathbf{w}}

\newcommand{\bA}{\bar{A}}
\newcommand{\F}{F_{\lambda}}

\newcommand{\cIe}{\mathscr{I}^e}
\newcommand{\Ie}{I^e}
\newcommand{\Xe}{X^e}

\newcommand{\brho}{\rho^*}
\newcommand{\lZ}{(\lambda\Z)^2}
\newcommand{\phil}{\varphi_{\lambda}}

\newcommand{\vI}{\langle e/2\rangle}
\newcommand{\vk}{\langle e \rangle}

\def\defn#1{\textbf{#1}}
\def\Fix#1{\mathrm{Fix}\,#1}
\def\mod#1{\hskip 15pt \left(\mathrm{mod}\;\,#1\right)}

\def\fl#1{\lfloor #1\rfloor}
\def\flsq#1{\langle #1\rangle}
\def\Bfl#1{\left\lfloor #1 \right\rfloor}
\def\ceil#1{\lceil #1\rceil}
\def\Bceil#1{\left\lceil #1\right\rceil}

\title{Asymptotics in a family of linked strip maps}
\author{Heather Reeve-Black \& Franco Vivaldi}
\address{School of Mathematical Sciences, Queen Mary, 
University of London, London E1 4NS, UK}
\email{h.reeve-black@qmul.ac.uk}
\email{f.vivaldi@qmul.ac.uk}
\date{\today}

\begin{document}

\begin{abstract}
We apply round-off to planar rotations, obtaining a one-parameter family of 
invertible maps of a two-dimensional lattice. 
As the angle of rotation approaches $\pi/2$, the fourth iterate of the map produces 
piecewise-rectilinear motion, which develops along the sides of convex polygons. 

We characterise the dynamics ---which resembles outer billiards of polygons---
as the concatenation of so-called strip maps, each providing an elementary perturbation of 
an underlying integrable system. Significantly, there are orbits which are subject to
an arbitrarily large number of these perturbations during a single revolution, resulting in the 
appearance of a novel discrete-space version of near-integrable Hamiltonian dynamics.

We study the asymptotic regime of the limiting integrable system analytically, and
numerically some features of its very rich near-integrable dynamics.
We unveil a dichotomy: there is one regime in which the nonlinearity tends to zero, 
and a second where it doesn't. In the latter case, numerical experiments suggest 
that the distribution of the periods of orbits is consistent with that of random dynamics; 
in the former case the fluctuations result in an intricate structure of resonances.
\end{abstract}

\maketitle

\section{Introduction} \label{sec:Introduction}

We study an asymptotic problem of planar rotations subject to round-off, 
which leads to a family of \defn{linked strip maps} on lattices.
These two-dimensional dynamical systems originate from repeated applications 
of a simple scattering mechanism, which we now describe.
We consider a line $l$ on the plane, which acts as a refracting surface
for a bundle of incoming parallel rays. As the rays cross the surface,
they change direction while remaining parallel.
The dynamics is provided by a piecewise-constant velocity field $\mathbf{v}$ 
which assumes the values $\bfv_1$ and $\bfv_2$ on the two half-planes 
determined by $l$ (see figure \ref{fig:Scattering}, left). 
We stipulate, say, that $\bfv(z)=\bfv_1$ on $l$.
The components of the vector field must agree in orientation: if $\mathbf{n}$ 
is any normal to $l$, we require the 
scalar products $\bfv_1\cdot \mathbf{n}$ and $\bfv_2\cdot \mathbf{n}$ to be 
non-zero and agree in sign. 
The orbits of the flow $\dot z=\bfv(z)$ are parallel rays joined at $l$.

Next we perturb the flow by turning it into a map:
\begin{equation}\label{eq:F}
F_\lambda:\R^2\to\R^2\qquad F_\lambda(z)=z+\lambda \bfv(z),
\end{equation}
where $\lambda>0$ is the perturbation parameter.
The map $F_\lambda$ is area-preserving, and it agrees with the time-$\lambda$ advance map of the flow,
except on the strip $\Sigma$ bounded by $l$ (included) and $\F^{-1}(l)$ (excluded). 
Upon crossing the line $l$, an orbit of $\F$ will jump from an orbit of the flow 
to a nearby orbit, which we interpret as a perturbation of the original 
motion (figure \ref{fig:Scattering}, right). 
If $d$ is the distance between $z$ and $l$ along the incoming field $\bfv_1$, then the resulting
perturbation is determined solely by the remainder of $d$ modulo $\Vert\lambda\bfv_1\Vert$. 
The map $\F$ is invertible precisely when the orthogonal component of the incoming 
field is the same as that of the outgoing field; we assume that this transversality condition holds.
By varying the parameter $\lambda$, we can adjust the width of $\Sigma$, hence the 
size of the perturbation. Variants of this construction will be considered later.

\begin{figure}[b]
        \centering
        \begin{minipage}{7cm}
          \centering
	  \includegraphics[scale=1.0]{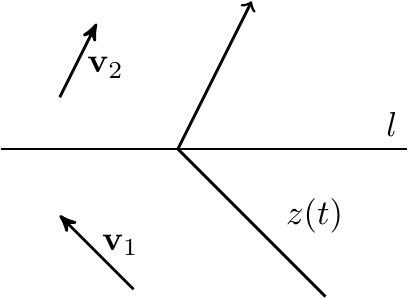}
        \end{minipage}
        \begin{minipage}{7cm}
	  \centering
	  \includegraphics[scale=1.0]{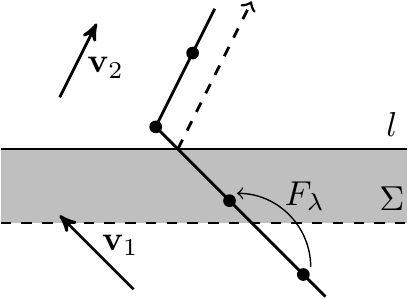} 
	\end{minipage}
        \caption{Rays crossing a surface. 
         Left: the flow; right: the map $\F$ of equation (\ref{eq:F}).}
        \label{fig:Scattering}
\end{figure}

To obtain non-trivial dynamics, we need repeated scattering. We consider a (finite or infinite) sequence 
$(l_1,l_2,\ldots)$ of lines, rays, and segments which partition the plane into convex domains.
We assume that on each domain there is a constant velocity field $\bfv_j$ and that 
the fields in any two adjacent domains satisfy the transversality conditions stated above. 
Then we have a flow defined over $\R^2$ except at the points where two or more lines 
$l_j$ intersect. 
The orbits of the flow are polygonal lines, whose vertices belong to the lines $l_j$. 

As above, we turn the flow into a map $F_\lambda$ which acts by translation by the vectors $\bfv_j$.
As a result, each line $l_j$ gives rise to an adjacent strip $\Sigma_j$, where flow and 
map disagree. In the regions where strips overlap, the map $F_\lambda$ may be defined, say,
by first ordering the strips and then applying the map of highest ranking. 
The map sending a point in one strip to the image point in the next strip is 
straightforward ---it only requires the computation of a remainder.  
So the strip-to-strip transit may be eliminated from the dynamics, and the map $F_\lambda$ 
becomes a map on $\bigcup\Sigma_j$, which we call a \defn{linked strip map}. 

Now let $\bbV$ be the smallest $\mathbb{Z}$-module containing all the 
vectors $\bfv_j$. Its dimension depends on the rational dependencies 
among these vectors, and in the simplest nontrivial case we have a 
two-dimensional lattice. 
The lattice $\bbV$ and its translates are invariant under $F_\lambda$ and we are interested in 
the restriction of $F_\lambda$ to these lattices, where
the $\lambda\to 0$ scaling of the vector field brings about a highly non-trivial 
`quasi-continuum' limit on these lattices.

We consider the case in which the orbits of the flow are closed polygons, 
which foliate the plane, giving an integrable Hamiltonian system.
The map $F_\lambda$ of equation (\ref{eq:F}) will be viewed as a perturbation of such
an integrable system. Note however, that the transversality condition introduced
above is no longer sufficient to guarantee that $F_\lambda$ is invertible or 
area-preserving (although it will be so in our case, due to time-reversal symmetry).

Integrable systems with polygonal invariant curves have been studied in the 
context of ultradiscrete (or tropical) dynamics \cite{Iwao,Nobe}.
Perturbed systems of the linked strip map type appear in outer billiards of convex 
polygons (see figure \ref{fig:StripMaps} (a) and \cite{Tabachnikov}). 
Here the analogous map is the second iterate of the 
outer billiard map for all points sufficiently far away from the origin.
This coincidence suggests a way of extending the map to the intersection of the 
strips, which are found near the central polygon. This programme was carried out 
by Schwartz \cite{Schwartz11}, resulting in the construction of the so-called 
the \textit{pinwheel map}.
Pinwheel maps may be viewed as a generalisation of outer billiard maps, in the sense 
that the unbounded orbits of the two systems are in one-to-one correspondence.

\begin{figure}[t]
        \centering
        \begin{minipage}{7cm}
          \centering
	  \includegraphics[scale=0.8]{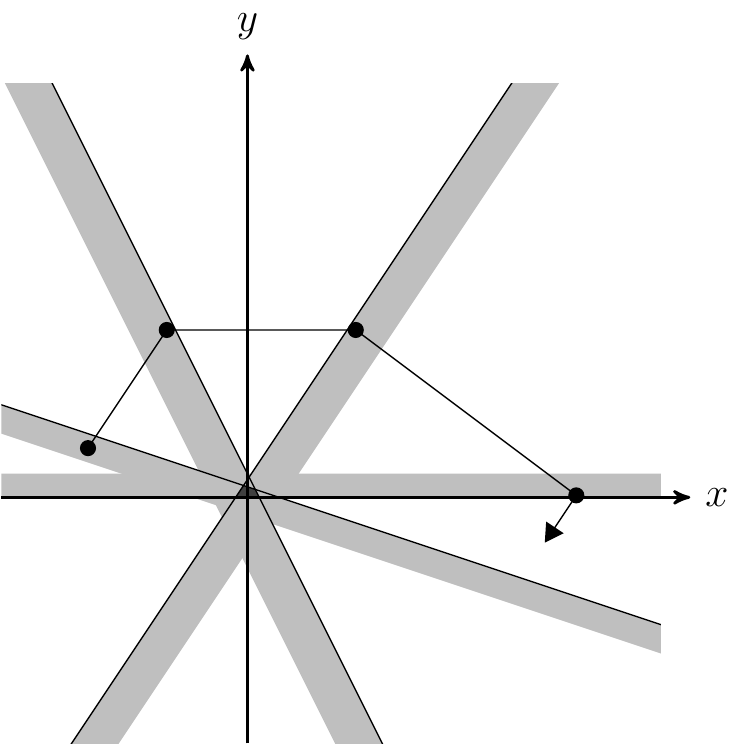} \\
	  (a) Outer billiard \\
        \end{minipage}
        \quad
        \begin{minipage}{7cm}
	  \centering
	  \includegraphics[scale=0.8]{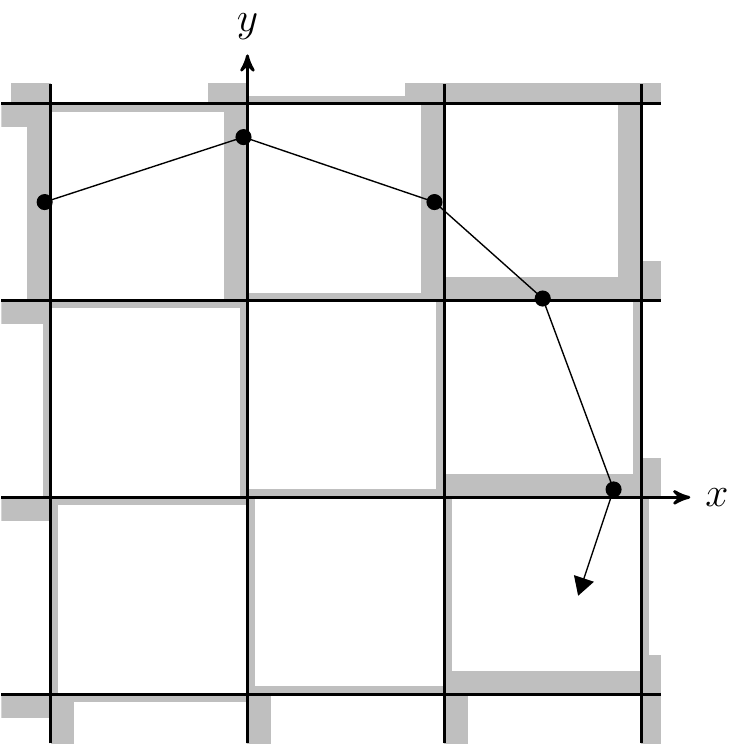} \\
	  (b) discretised rotations \\
	\end{minipage}
        \caption{Schematic representation of the linked strip map construction for
(a) the outer billiard of a polygon, and (b) the discretised rotation.
The latter construction features infinitely many strips.}
        \label{fig:StripMaps}
\end{figure}

In this paper we explore a specific invertible linked strip map of a two-dimensional 
lattice, which results from discretising the space of a one-parameter family
of planar rotations via a round-off procedure. 
The novel feature of this model ---to be defined in the next section--- is that 
the number of lines $l_j$ is infinite, and moreover the invariant curves of the 
underlying integrable system, which are convex polygons, have an arbitrarily large 
number of sides.
When the integrable flow is turned into a map, the orbits located sufficiently 
far away from the origin meet an arbitrarily large number of strips during a 
single revolution.
This results in a perturbation of increasing complexity---a feature which
cannot be achieved in outer billiards of polygons without changing the shape 
of the billiard. This asymptotic regime is the primary interest of this paper.

\subsection{The model}\label{section:TheModel}

The map under study is given by
\begin{equation} \label{eq:Map}
F:\,\Z^2\rightarrow\Z^2 \hskip 20pt
F(x,y)=(\lfloor \lambda x \rfloor - y, \,x)
\hskip 20pt \lambda=2\cos(2\pi\nu)
\end{equation}
where $\lfloor \cdot \rfloor$ is the floor function 
---the largest integer not exceeding its argument.
If we remove the floor function in (\ref{eq:Map}), then
we obtain a one-parameter family of linear maps of the plane, 
which, for $|\lambda|<2$, are linearly conjugate to a rotation by the angle $2\pi\nu$ ($\nu$ is the rotation number).
The floor function discretises the space, rounding the image point 
$(\lambda x-y,x)$ to the nearest lattice point on the left. 
This map is invertible as well as reversible. (Note that a generic discretised 
rotation with round-off will not have either property.) 
In this system the discretisation length is fixed, and the 
limit of vanishing discretisation corresponds to motions at infinity.

The model (\ref{eq:Map}) has strong arithmetical features.
The original motivation for its study came from dynamical systems 
\cite{Vivaldi94b,LowensteinHatjispyrosVivaldi,LowensteinVivaldi98,
LowensteinVivaldi00,BosioVivaldi,VivaldiVladimirov,KouptsovLowensteinVivaldi02};
subsequently, this model was studied in number theory in the area of shift 
radix systems \cite{AkiyamaBrunottePethoThuswaldner,AkiyamaBrunottePethoSteiner06,AkiyamaBrunottePethoSteiner08}.

The dynamics of (\ref{eq:Map}) is poorly understood. It has been conjectured that 
for all real $\lambda$ with $|\lambda|<2$ all orbits of $F$ are periodic 
\cite{AkiyamaBrunottePethoThuswaldner,Vivaldi06}, but this conjecture has been proved
for only \textit{eight} values of $\lambda$ (besides three trivial cases $\lambda=0,\pm1$ 
where the map $F$ is of finite order). 
These parameters correspond to the rational values of the rotation number $\nu$ 
for which $\lambda$ is a quadratic irrational
\cite{LowensteinHatjispyrosVivaldi,KouptsovLowensteinVivaldi02,AkiyamaBrunottePethoSteiner08}.
The proof relies on an embedding in a two dimensional torus, where the system becomes 
an exactly renormalisable piecewise isometry, which is a zero-entropy system.
The lack of renormalisability for algebraic parameters of degree three or
higher has so far prevented further progress.

The map (\ref{eq:Map}) is known to have infinitely many periodic orbits for all 
parameter values \cite{AkiyamaPetho}; however these orbits have zero density. 
The analysis of rational $\lambda$-values (for which $\nu$ is irrational) 
gives some insight into the difficulties one encounters to establish periodicity.
In the simplest scenario, where the denominator of $\lambda$ is the power of a prime $p$,
the map $F$ admits a dense embedding in the ring $\Z_p$ of $p$-adic integers.
The embedding dynamics has positive entropy, being the composition of an isometry
and a full shift \cite{BosioVivaldi,VivaldiVladimirov}. 
The periodic orbits of the round-off map form a thin subset of all periodic orbits of the
embedding system: establishing their properties is predictably difficult.

In this work we consider the limit $\lambda\to 0$, 
which we refer to as the \defn{integrable limit}. 
As pointed out above, at $\lambda=0$ the dynamics is trivial: there is no round-off 
and the mapping has order four ($\nu=1/4$).
In this parameter regime there is a new natural embedding of $F$, this time into the plane,
and a new dynamical mechanism, namely a discrete-space version of linked strip maps.
Some properties of this systems were studied in \cite{ReeveBlackVivaldi}, and we
will review them below.

\begin{figure}[t]
        \centering
        \begin{minipage}{7cm}
          \centering
	  \includegraphics[scale=0.35]{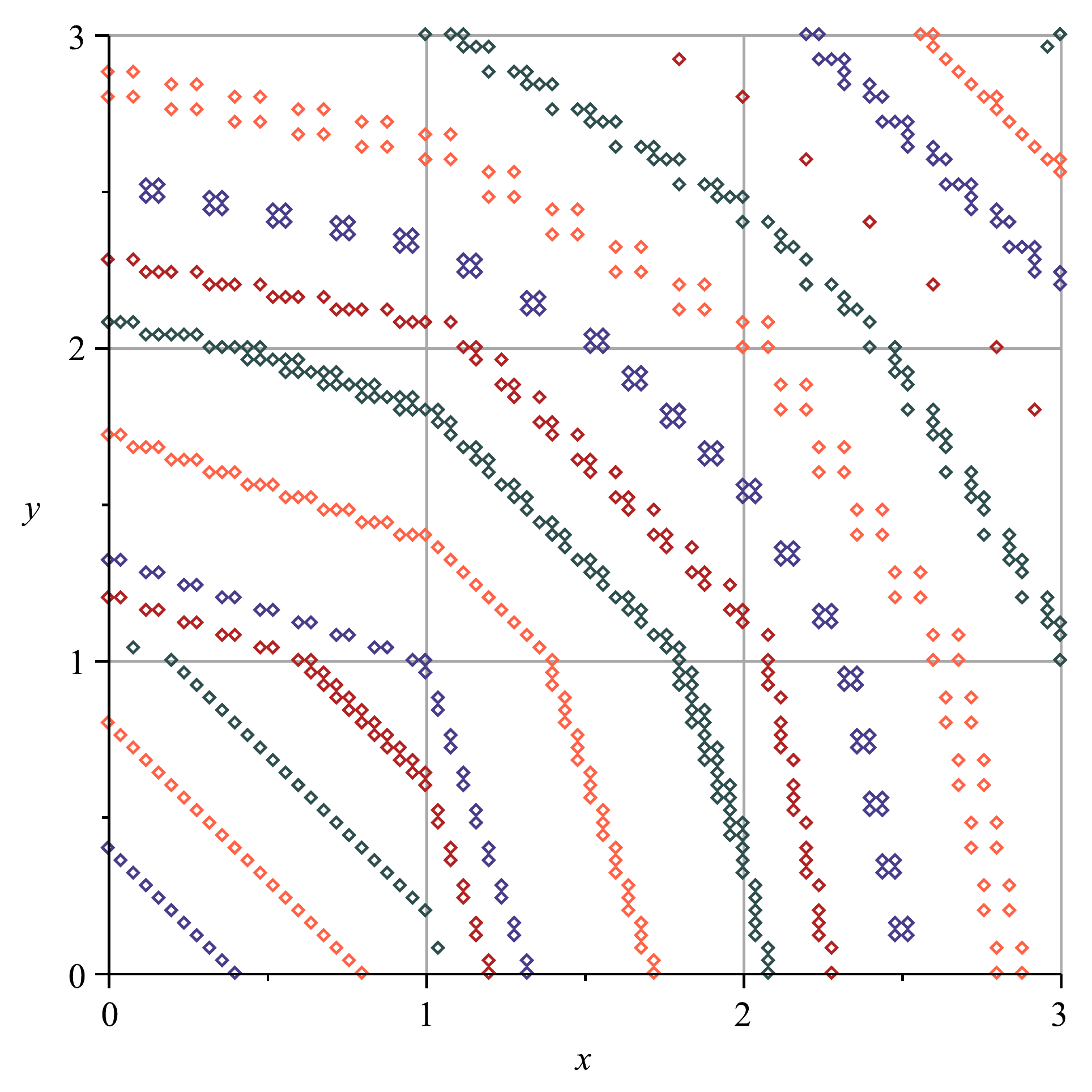} \\
	  (a) $\; \lambda=1/25$ \\
        \end{minipage}
        \quad
        \begin{minipage}{7cm}
	  \centering
	  \includegraphics[scale=0.35]{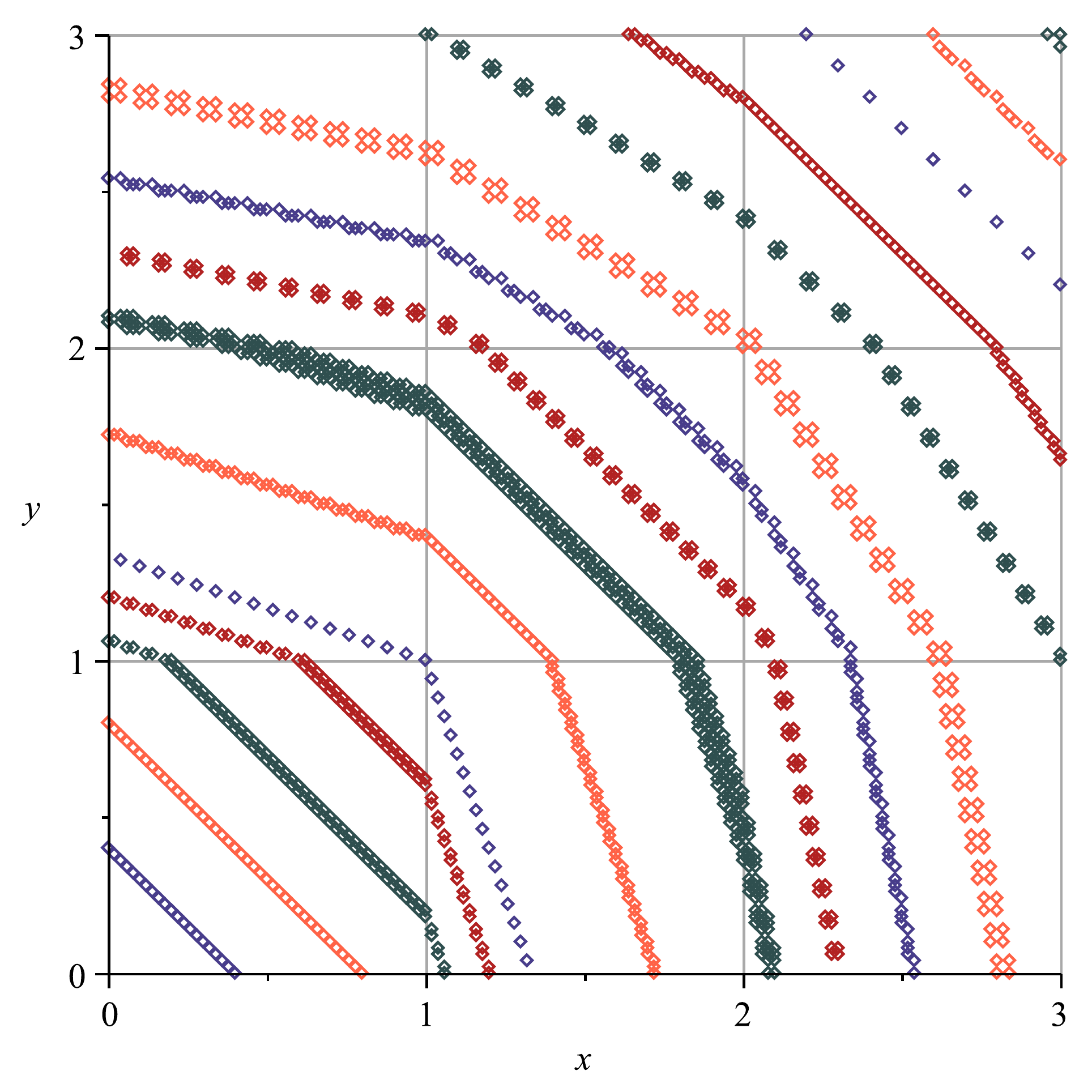} \\
	  (b) $\; \lambda=1/50$ \\
	\end{minipage}
        \caption{Some periodic orbits of the rescaled map $F_{\lambda}$ for small values of 
the parameter $\lambda$. The view is restricted to the first quadrant.
The lattice spacing is such that each unit distance contains 
$1/\lambda$ lattice points.}
        \label{fig:PolygonalOrbits}
\end{figure}

The connection with linked strip maps stems from the fact that, as $\lambda$ approaches 0,
the iterated map $F^4$ is close to the identity,
suggesting the 
introduction of a \textit{discrete vector field} on the scaled lattice $(\lambda\mathbb{Z})^2$:
$$ 
  \bfv_\lambda(z) = F_\lambda^4(z) - z,
$$
where $F_\lambda$ is a scaled version of (\ref{eq:Map}):
\begin{equation}\label{eq:F_lambda}
\F: \lZ \to \lZ 
 \hskip 40pt
 \F(z)=\lambda F(z/\lambda)
  \hskip 40pt
 \lambda>0.
\end{equation}
For $\lambda$ sufficiently close to 0, the module $\bbV$ of $\bfv_\lambda$ is $(\lambda\mathbb{Z})^2$.
The parameter $\lambda$ controls not only the rotation number $\nu$, but also the lattice spacing.
It turns out that, as $\lambda\to 0$, the field $\bfv_\lambda$ becomes piecewise-constant, 
albeit in a non-uniform manner. The domains of rectilinear motion become a regular array of squares, 
bounded by an infinite array of orthogonal lines $l_j$.

More precisely, there is an auxiliary integrable Hamiltonian vector field $\bfw$ on $\R^2$, 
with the property that, if $\lambda$ is sufficiently small, then $\lambda\bfw$ 
agrees with $\bfv_\lambda$ over any bounded domain of the plane, apart from regions of
negligible total measure. This field is given by:
\begin{equation}\label{eq:w}
 \bfw: \; \R^2 \to \Z^2
\hskip 40pt
\bfw(x,y)=(2\fl{ y }+1,-(2\fl{ x }+1)).
\end{equation}
The convex domains on which $\bfw$ is constant are bounded by the array of
orthogonal lines $x=i$ and $y=j$, for integer $i,j$.
These domains are translated unit squares (called \defn{boxes})
\begin{equation}\label{eq:B_mn}
 B_{m,n} = (m,n) + [0,1)^2 \hskip 40pt m,n\in\Z
\end{equation}
and the field across adjacent boxes satisfies the transversality condition mentioned above. 
The module $\bbV$ of $\bfw$ is generated by $(1,1)$ and $(1,-1)$, 
and hence has index 2 in $\Z^2$.
The corresponding Hamiltonian $\cH$ is given by
\begin{equation} \label{eq:Hamiltonian}
 \cH: \; \R^2 \; \to \R
 \hskip 40pt
 \cH(x,y) = P(x)+P(y),
\end{equation}
where $P$ is the piecewise-affine function
\begin{equation}\label{eq:P}
P:\R\to \R \hskip 40pt P(x)=\fl{x}^2+(2\fl{x}+1)\{x\},
\end{equation}
and $\{x\}$ denotes the fractional part of $x$.
The orbits of $\cH$ are convex polygons, with an increasing number of sides 
as one moves away from the origin; near the origin they are squares, while at 
infinity they approach circles.
For small $\lambda$, the composite map $F_\lambda^4$ is equal to the 
time-$\lambda$ advance map of the flow along the edges of the polygons. 
The perturbation occurs near the vertices, where the orbits are not differentiable.

This structure differs from the pinwheel map in an important respect.
Not only is the number of domains over which the field is constant infinite, 
but in addition, as the distance from the origin increases, the number of 
components encountered by individual orbits also increases without bounds.
In our model the auxiliary vector field $\lambda\bfw$ is not necessarily equal 
to the discrete vector field $\bfv_\lambda$---the two may differ within the strips.
This discrepancy ---while creating additional complications--- admits a simple description.

Thus, in the integrable limit, the rescaled map $F_\lambda$ is a perturbation of an
integrable Hamiltonian system, but the integrable system is no longer a rotation, 
and the perturbation is no longer caused by round-off. 
Furthermore, the integrable system is nonlinear, i.e., its time-advance 
map satisfies a twist condition. When the perturbation is switched on, a 
discrete-space version of near-integrable Hamiltonian dynamics emerges on 
the lattice (see figure \ref{fig:PolygonalOrbits}).

\medskip

\subsection{Outline of the paper \& main results}

The study of the rescaled map $\F$ at infinity begins with the analysis of
the underlying integrable system.
The invariant polygons of the Hamiltonian system (\ref{eq:Hamiltonian}) 
form countably many foliations, called \defn{polygon classes}. All polygons in a class 
have the same number of sides, and the classes are indexed by the set $\cE$ of natural 
numbers which can be represented as the sum of two squares:
\begin{equation} \label{eq:cE}
 \cE = \{e_i \, : \; i\geqslant 0\} = \{0,1,2,4,5,8,9,10,13,16,17,\dots \} .
\end{equation}
As $e_i\to\infty$, the corresponding foliation becomes a thin polygonal 
annulus of approximate radius $\sqrt{e_i}$.
Further definitions and results from \cite{ReeveBlackVivaldi} 
will be reviewed in section \ref{sec:IntegrableFlow}.

\medskip

Then we construct a Poincar\'{e} return map for the integrable system.
We consider a the time-advance map of the flow and its first return to a thin 
strip $X$ placed along the symmetry line $x=y$. 
The width of the strip is equal to the magnitude of the the field $\lambda\bfw$
and since a point and its translate by $\lambda\bfw$ are naturally identified, 
we have a return map of a cylinder. Its behaviour is determined by the period 
of the flow, which varies in a piecewise-affine manner within each polygon class.

\medskip

In section \ref{sec:IntegrableAsymptotics} we study the return map at infinity.
For each class there is a change of coordinates which conjugates the unperturbed 
return map to a linear twist map on the cylinder.
The twist condition depends on the derivative of the period function for that class,
and as the distance from the origin increases, this function undergoes damped oscillations.
Under a suitable scaling, the period function converges to a continuous function with
discontinuous first derivative. This in turn leads to a discontinuity in the asymptotic 
behaviour of the twist map. 

More precisely, in section \ref{sec:IntegrableAsymptotics} we shall prove the following 
result, which appears in the text as two separate theorems.

\begin{thm_nonumber}[Theorem \ref{thm:Omega_e}, page \pageref{thm:Omega_e} \& 
theorem \ref{thm:Tprime_asymptotics}, page \pageref{thm:Tprime_asymptotics}]
Associated with each polygon class, indexed by $e\in\cE$, there is a change of coordinates
which conjugates the unperturbed return map to the linear twist map $T^e$, given by
\begin{equation}\label{eq:T^e}
 T^e : \X \to \X
    \hskip 40pt 
 T^e(\theta,\rho) = \left( \theta + \kappa\rho, \rho \right),
\end{equation}
where $\X$ is the unit cylinder:
 $$
\X=\bbS^1\times\R,
$$
and $\kappa=\kappa(e)$ is the twist.
Furthermore, as $e\to\infty$, the limiting behaviour of $\kappa(e)$ is singular:
\begin{displaymath}
 \kappa(e) \to \left\{ \begin{array}{ll} 4 \quad & \{\sqrt{e}\}=0 \\ 0 \quad & \mbox{otherwise.} \end{array}\right.
\end{displaymath}
\end{thm_nonumber}

\medskip

Typically the twist $\kappa$ converges to zero, and hence the return map converges 
(non-uniformly) to the identity. However, for the polygon classes for which $e$ is 
a perfect squares (a thin sub-sequence of the sequence (\ref{eq:cE})), the twist 
approaches a positive value.

Finally, in section \ref{sec:PerturbedDynamics}, we explore numerically the perturbed 
dynamics at infinity, tracking exactly long orbits of $F$ using integer arithmetic. 
We study some features of the first-return map to the strip $X$. This is now a perturbed 
twist map on a lattice. The twist condition is given by the theorem above, and the 
perturbation results from the cumulative effect of a large number of strip maps. 
In the annuli where the twist $\kappa(e)$ tends to zero, the form of the perturbation 
is laid bare, and we find delicate resonance structures superimposed to
low-amplitude fluctuations (see figure \ref{fig:ResonancesCloseUp}). 
By contrast, where the twist persists the phase portrait is featureless, suggesting 
a statistical approach.

\begin{figure}[t] 
        \centering
        \includegraphics[scale=0.25]{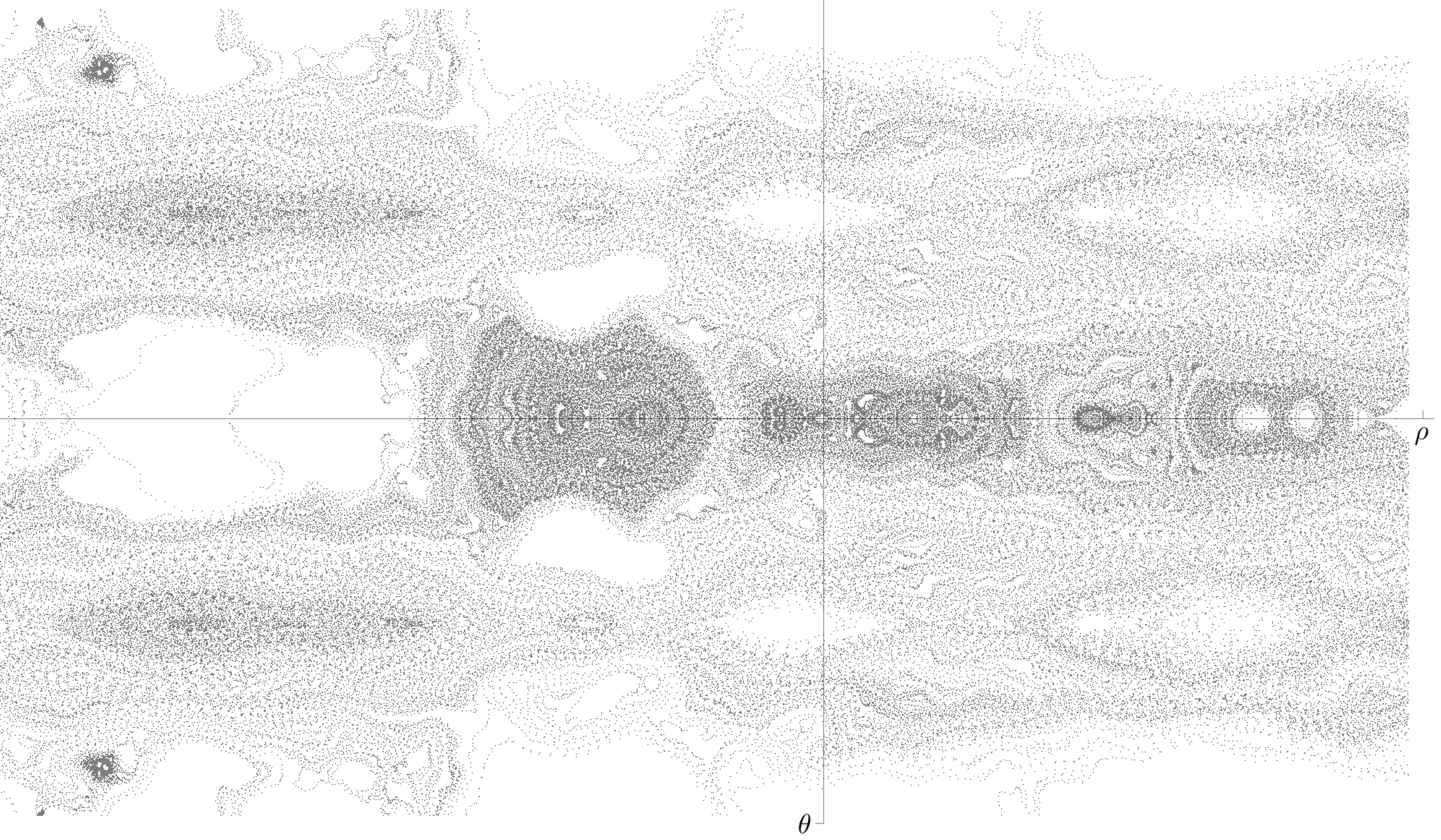} \\
        \caption{A pixel-plot of a primary resonance for $e=160234\approx400.3^2$ and $\lambda\approx 3.5\times 10^{-9}$. 
        The plot shows part of a large number of symmetric orbits of $\Phi$ in the cylindrical coordinates
        $(\theta,\rho)\in\X=\bbS^1\times\R$. 
        The resolution is such that the width of the cylinder consists of approximately 560 lattice sites.
        There are parts of the cylinder which are filled with symmetric orbits, 
        and others which are devoid of them, but no sharp boundary between the two. }
        \label{fig:ResonancesCloseUp}
\end{figure}

 
Thus we consider the distribution of the periods of the periodic orbits.
Since our system is reversible, it is the composition of two orientation-reversing 
involutions. The probabilistic model against which to compare the dynamical data 
was developed in \cite{RobertsVivaldi09}, where it was shown that for a suitable 
scaling of the period, the expected period distribution function of a random 
reversible map on $N$ points converges pointwise as $N\rightarrow\infty$. 
The limit function is a gamma (or Erlang) distribution
\begin{equation} \label{eq:R(x)}
 \cR(x) = 1 - e^{-x}(1+x).
\end{equation}
The scaling factor depends on the fraction of the phase space occupied by the 
fixed sets of the two involutions, see \cite[Theorem A]{RobertsVivaldi09}.

Numerical experiments with various classes of maps over finite fields 
(polynomial and rational maps of affine and projective spaces and K3 surfaces), 
which have strong pseudo-random features, have exposed the ubiquity of 
(\ref{eq:R(x)}) \cite{RobertsVivaldi05,Hutz}. 
This distribution was subsequently found also for the periodic orbits of 
a zero-entropy map on lattices \cite{NeumarkerRobertsVivaldi}, suggesting that
the emergence of (\ref{eq:R(x)}) requires only mild statistical properties.

The period distribution function is well-defined in our model, because the return map 
is equivariant with respect to a group of lattice translations, and so the phase space 
is a discrete torus.
In section \ref{sec:PerturbedDynamics} we provide numerical evidence that 
the limit $\cR(x)$ is attained along the sub-sequence of polygon classes which 
correspond to perfect squares, where the nonlinearity of the return map 
persists at infinity. 
These findings suggest that a sufficiently large number of linked strip maps 
can produce a form of `pseudochaos'.

\section{The integrable flow} \label{sec:IntegrableFlow}

We describe the properties of the Hamiltonian $\cH$ of equation 
\eqref{eq:Hamiltonian}, which was introduced in \cite{ReeveBlackVivaldi}.

The function $\cH$ is continuous, piecewise-affine, and differentiable on 
$\R^2\setminus \Delta$, where $\Delta$ is the set of orthogonal lines given 
by 
\begin{equation}\label{eq:Delta}
\Delta=\{(x,y)\in\R^2 \, : \; (x-\fl{x})(y-\fl{y})=0\}.
\end{equation}
The corresponding Hamiltonian vector field, wherever it is defined, 
is equal to the auxiliary vector field $\bfw$ of equation \eqref{eq:w}, 
and is constant on the collection of boxes $B_{m,n}$ of equation \eqref{eq:B_mn}.
The set $\Delta$ is the boundary of the boxes $B_{m,n}$.
We denote the value of $\bfw$ on $B_{m,n}$ as
\begin{equation} \label{eq:w_mn}
 \bfw_{m,n}=(2n+1,-(2m+1)). 
\end{equation}

The orbits ---the level sets of $\cH$--- are \defn{convex polygons} \cite[Theorem 2]{ReeveBlackVivaldi},
denoted by 
\begin{equation*} 
 \Pi(\alpha) = \{z\in\R^2 \, : \; \cH(z) = \alpha \} \hskip 40pt \alpha\in\R^+.
\end{equation*}
The vertices of a polygon $\Pi(\alpha)$ occur at its intersections with the set $\Delta$, 
and all polygons are invariant under the dihedral group $D_4$, 
generated by the two orientation-reversing involutions
\begin{equation}\label{eq:Dihedral}
 G:\quad (x,y) \mapsto (y,x) 
\hskip 40pt
 G':\quad (x,y) \mapsto (x,-y).
\end{equation}

There is a natural partition of the foliation $\Pi(\alpha)$ into \defn{polygon classes},
separated by \defn{critical polygons}, which contain integer lattice points.
Since $\cH(x,y)=x^2+y^2$ whenever $x$ and $y$ are integer,
a polygon $\Pi(\alpha)$ is critical precisely when $\alpha$ belongs to the set
$\cE$ of non-negative integers which are representable as the sum of two squares 
(see equation \eqref{eq:cE}), which we call \defn{critical numbers}.
The corresponding family of \defn{critical intervals}---the open intervals between 
consecutive critical numbers---is given by
\begin{equation}\label{eq:Ie}
 \cI^{e_i} = (e_i,e_{i+1})\qquad e_i\in\cE.
\end{equation}

\begin{figure}[t]
\centering
\includegraphics[scale=0.35]{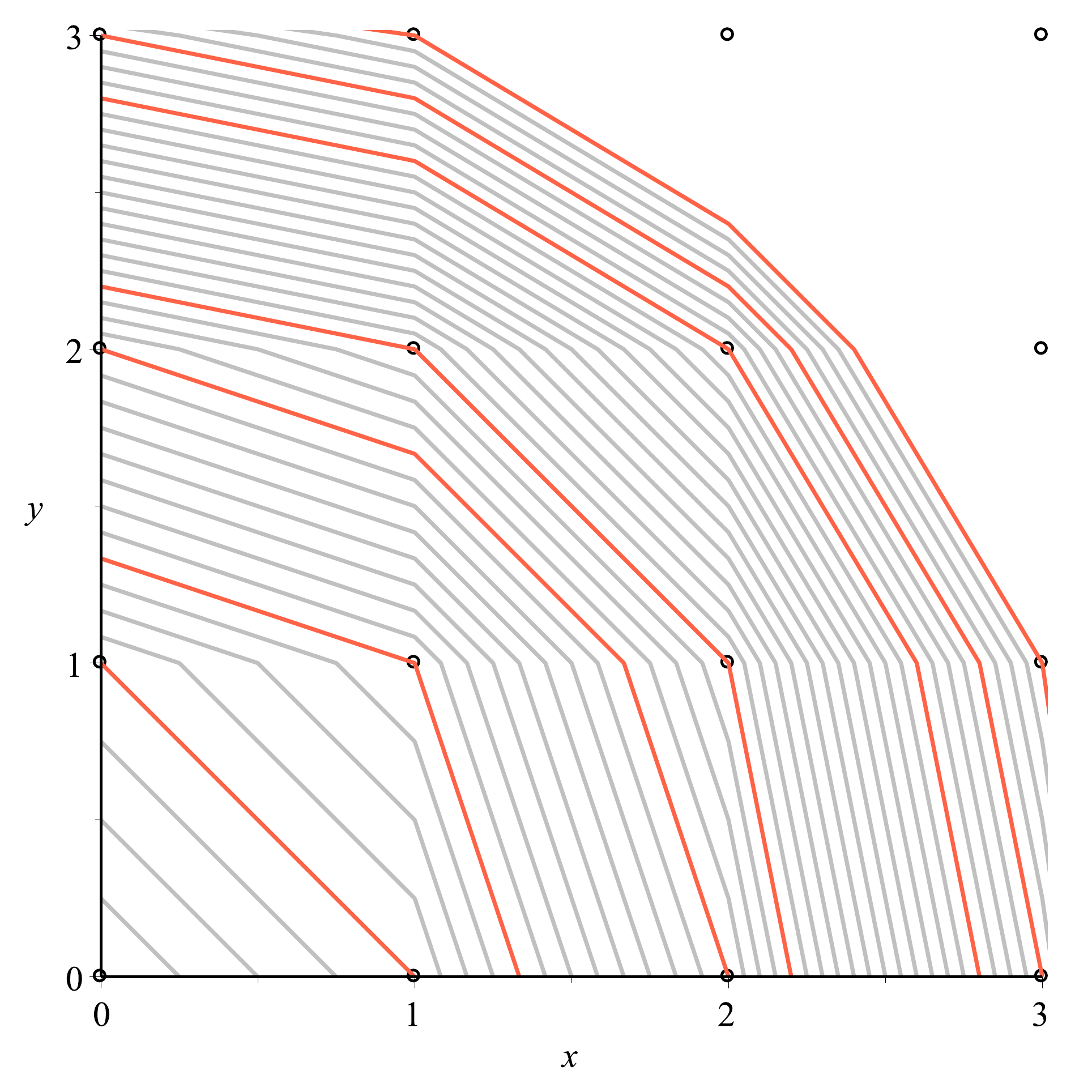}
\caption{A selection of polygons $\cH(x,y)=\alpha$ for values of $\alpha$ in the interval $[0,10]$.
The critical polygons are shown in red.
The polygon classes are the annuli bounded between pairs of adjacent critical polygons.}
\label{fig:polygon_classes}
\end{figure}

The critical polygons act as separatrices, delineating a countable set of concentric, open, polygonal annuli,
which constitute the polygon classes (see figure \ref{fig:polygon_classes}).
We index the polygon classes using the set of critical numbers, 
and associate the critical number $e\in\cE$ with the polygon class $\cH^{-1}(\cIe)$.
All polygons in a given class have the same set of edge vectors,
i.e., intersect the same collection of boxes $B_{m,n}$.

\subsection{Nonlinearity} \label{sec:Nonlinearity}

The Hamiltonian $\cH$ is nonlinear, since the periods of orbits depend
on the initial conditions. We consider the \defn{period function} 
$\alpha\mapsto \cT(\alpha)$, which gives the period of the orbit $\Pi(\alpha)$
as a function of $\alpha>0$.
Adopting the short-hand notation
\begin{equation}\label{eq:ShortHand}
\langle x\rangle:=\lfloor\sqrt{x}\rfloor
\end{equation}
we have the following expression for $\cT(\alpha)$.

\begin{proposition} \label{prop:T(alpha)}
 Let $\alpha > 0$. 
Then the period $\cT(\alpha)$ of the Hamiltonian flow on the polygon $\Pi(\alpha)$ is given by
\begin{equation} \label{eq:cT(alpha)}
 \frac{\cT(\alpha)}{8} = \frac{P^{-1}(\alpha/2)}{2\flsq{\alpha/2}+1} 
    -2 \sum_{n=\flsq{\alpha/2}+1}^{\flsq{\alpha}} \frac{P^{-1}(\alpha - n^2)}{4n^2-1}
\end{equation}
(where if $\flsq{\alpha/2}=\flsq{\alpha}$ the sum should be understood to be empty).
\end{proposition}

Note that the function $P$ of equation \eqref{eq:P} can only be inverted up to sign, 
so we write $P^{-1}$ to denote the function
\begin{equation}\label{def:Pinv}
 P^{-1}: \R_{\geqslant 0} \to \R_{\geqslant 0} 
\hskip 40pt 
x \mapsto \frac{x + \flsq{x} (1+\flsq{x})}{2\flsq{x}+1},
\end{equation}
which satisfies $(P^{-1}\circ P)(x) = |x|$.

\begin{proof}
Take $\alpha>0$. By the eight-fold symmetry of the level sets of $\cH$ (see equation \eqref{eq:Dihedral}),
it suffices to consider the intersection of the polygon $\Pi(\alpha)$ with the first octant. 
The point $(y,y)\in\Pi(\alpha)$ where the polygon intersects the positive half of the symmetry line $\Fix{G}$
satisfies 
 $$ y = P^{-1}(\alpha/2), $$
whereas the point $(x,0)\in\Pi(\alpha)$ where the polygon intersects the positive $x$-axis satisfies 
 $$ x = P^{-1}(\alpha). $$
We consider the time taken to flow between these two points.

From the piecewise-affine form of the function $P$, it is straightforward to show that
\begin{equation} \label{eq:flPinv}
 \flsq{x} = \fl{P^{-1}(x)}  \hskip 40pt  x\in\R^+,
\end{equation}
so that by the definition of $x$ and $y$:
 $$ \fl{y}=\flsq{\alpha/2} \hskip 40pt \fl{x}=\flsq{\alpha}. $$
To proceed, we partition the $y$-distance between the two points $(y,y)$ and $(x,0)$ into a sequence of distances $d(n)$,
and corresponding flow-times $t(n)$, for $n=\flsq{\alpha/2},\dots,\flsq{\alpha}$. 
If $\flsq{\alpha}=\flsq{\alpha/2}$ the partition consists of a single element: 
we simply let $d(\flsq{\alpha/2}) = P^{-1}(\alpha/2)$, and $t(\flsq{\alpha/2})$ is the time 
taken for the flow on $\Pi(\alpha)$ to move between the symmetry lines $x=y$ 
and $y=0$.

Suppose now that $\flsq{\alpha/2}<\flsq{\alpha}$.
Note that the $y$-coordinate of the vertex of $\Pi(\alpha)$ which lies in the first octant and has $x$-coordinate $x=n$ is given by
 $$ P^{-1}(\alpha - n^2). $$
(Such vertices do not exist if $\flsq{\alpha}=\flsq{\alpha/2}$.)
Then we let $d(\flsq{\alpha})$ denote the $y$-distance between the points where $\Pi(\alpha)$ 
intersects the symmetry line $x=y$ and the line $x=\flsq{\alpha/2}+1$:
 $$ d(\flsq{\alpha/2}) = P^{-1}(\alpha/2) - P^{-1}(\alpha - (\flsq{\alpha/2}+1)^2), $$
and $d(\flsq{\alpha})$ denote the $y$-distance between the points where $\Pi(\alpha)$ 
intersects the line $x=\flsq{\alpha}$ and the symmetry line $y=0$:
 $$ d(\flsq{\alpha}) = P^{-1}(\alpha - \flsq{\alpha}^2). $$
For any $n$ with $\flsq{\alpha/2}+1\leqslant n \leqslant \flsq{\alpha}-1$, 
$d(n)$ is the $y$-distance between the points where $\Pi(\alpha)$ 
intersects the lines $x=n$ and $x=n+1$:
 $$ d(n) = P^{-1}(\alpha - n^2) -P^{-1}(\alpha - (n+1)^2). $$
Similarly, $t(\flsq{\alpha/2})$ is the time taken for the flow on $\Pi(\alpha)$ 
to move between the symmetry line $x=y$ and $x=\flsq{\alpha/2}+1$, $t(\flsq{\alpha/2})$ 
is the time taken to move between the line $x=\flsq{\alpha}$ and the symmetry 
line $y=0$, and for $\flsq{\alpha/2}+1\leqslant n \leqslant \flsq{\alpha}-1$, $t(n)$ is the time 
taken to flow between the lines $x=n$ and $x=n+1$.

With this notation, and using the symmetry of the flow (\ref{eq:Dihedral}), 
the period $\cT(\alpha)$ satisfies
\begin{equation} \label{eq:T(r)_sum_t(n)}
 \frac{\cT(\alpha)}{8} = \sum_{n=\flsq{\alpha/2}}^{\flsq{\alpha}} t(n).
\end{equation}

For any $n\geqslant 0$, the auxiliary vector field has constant $y$-component 
between the lines $x=n$ and $x=n+1$, given by $-(2n+1)$ (see equation (\ref{eq:w})).
Hence the times $t(n)$ and the distances $d(n)$ are related by
\begin{equation} \label{eq:t(n)}
 t(n) = \frac{d(n)}{2n+1} \hskip 40pt \flsq{\alpha/2}\leqslant n\leqslant \flsq{\alpha}.
\end{equation}

If $\flsq{\alpha/2}=\flsq{\alpha}$ then the result follows from the definition of $d(\flsq{\alpha/2})$. 
If $\flsq{\alpha/2}<\flsq{\alpha}$, substituting (\ref{eq:t(n)}) and the definition of the $d(n)$ into (\ref{eq:T(r)_sum_t(n)}) gives:
\begin{align*}
 \frac{\cT(\alpha)}{8} &= \sum_{n=\flsq{\alpha/2}}^{\flsq{\alpha}} \frac{d(n)}{2n+1} \\
 &= \frac{P^{-1}(\alpha/2)}{2\flsq{\alpha/2}+1} + \sum_{n=\flsq{\alpha/2}+1}^{\flsq{\alpha}} \frac{P^{-1}(\alpha - n^2)}{2n+1}
 - \sum_{n=\flsq{\alpha/2}}^{\flsq{\alpha}-1} \frac{P^{-1}(\alpha - (n+1)^2)}{2n+1} \\
&= \frac{P^{-1}(\alpha/2)}{2\flsq{\alpha/2}+1} -2 \sum_{n=\flsq{\alpha/2}+1}^{\flsq{\alpha}} \frac{P^{-1}(\alpha - n^2)}{4n^2-1},
\end{align*}
as required.
\end{proof}

\medskip

Now we consider the derivative of the period function $\cT$.
If the derivative is non-zero, then the flow is nonlinear.

\begin{proposition} \label{prop:T'(alpha)}
Let $e\in\cE$ be a critical number and let $\alpha\in\cIe$ be some value in the critical interval associated with $e$.
Then the derivative $\cT^{\prime}(\alpha)$ of the period of the Hamiltonian flow on the polygon $\Pi(\alpha)$ is given by
\begin{equation} \label{eq:Tprime(alpha)}
 \frac{\cT^{\prime}(\alpha)}{4} = \frac{1}{(2\vI+1)^2} -4 \sum_{n=\vI+1}^{\vk} \frac{1}{(4n^2-1)(2\flsq{e - n^2}+1)},
\end{equation}
(where if $\vI=\vk$ the sum should be understood to be empty).
Thus the derivative $\cT^{\prime}(\alpha)$ is piecewise constant on the critical intervals $\cIe$.
\end{proposition}

\begin{proof}
We differentiate the expression \eqref{eq:cT(alpha)} of proposition \ref{prop:T(alpha)} with respect to $\alpha$.
The function $P^{-1}(x)$ of \eqref{def:Pinv} is differentiable at every $x$ which is not a perfect square, with
 $$ \frac{dP^{-1}(x)}{dx} = \frac{1}{2\flsq{x}+1} \hskip 40pt \sqrt{x} \notin\Z. $$
Letting $x=\alpha - n^2$, we have that the derivative exists whenever $\alpha$ cannot be expressed as a sum of squares:
 $$ \frac{dP^{-1}(\alpha - n^2)}{dx} = \frac{1}{2\flsq{\alpha - n^2}+1} \hskip 40pt \alpha\notin\cE. $$
Furthermore, the floor in the denominator of the above is constant between two consecutive sums of squares:
 $$ \flsq{\alpha - n^2} = \flsq{e - n^2} \hskip 40pt \alpha\in\cIe, $$
so that the derivative is constant on each of the critical intervals $\cIe$.
Similarly the bounds of the sum in \eqref{eq:cT(alpha)} are a function of $e$ only:
 $$ \flsq{\alpha/2}=\vI \hskip 40pt \flsq{\alpha}=\vk \hskip 40pt \alpha\in\cIe. $$
Thus if $\alpha\in\cIe$ is non-critical, then $\cT(\alpha)$ is differentiable and
\begin{equation*} 
 \frac{\cT^{\prime}(\alpha)}{4} = \frac{1}{(2\vI+1)^2} -4 \sum_{n=\vI+1}^{\vk} \frac{1}{(4n^2-1)(2\flsq{e - n^2}+1)},
\end{equation*}
as required.
\end{proof}

Note that the derivative $\cT^{\prime}(\alpha)$ is not defined at the critical numbers.
In what follows, we write $\cT^{\prime}(e)$ to denote the value of $\cT^{\prime}(\alpha)$ on $\cIe$.

\section{The integrable asymptotics} \label{sec:IntegrableAsymptotics}

In this section we define a time-advance map $\cF$ of the flow of the hamiltonian $\cH$. 
Here $\lambda$ is a scaling parameter which controls the speed of the flow; it will be
used to match the integrable orbits with those of the rescaled discretised rotation $\F$ 
(\ref{eq:F_lambda}).
The map $\cF$ has a natural recurrence time, and we consider its first return to a 
Poincar\'{e} section aligned along the symmetry line $x=y$.
We find that, under a suitable change of coordinates, the first return map
within each polygon class is a linear twist map on a cylinder.

Then we consider the behaviour of the Hamiltonian system at infinity, 
where the index $e\in\cE$ diverges.  
In this limit the number of discontinuities experienced by an orbit 
(i.e., the number of vertices of the polygons) is unbounded, 
whilst the magnitude of these discontinuities becomes arbitrarily small.
We find that the limiting behaviour is a dichotomy. 
Typically the limiting flow is linear, like the underlying rigid rotation: 
however, for a certain subsequence of values of $e$ the nonlinearity persists.

\subsection{The unperturbed return map} \label{sec:UnperturbedReturnMap}

Let $\cF$ be the following map:
\begin{equation} \label{eq:cF}
 \cF : \R^2 \rightarrow \R^2  \hskip 40pt \cF(z) = \varphi_{(\lambda - \cT(z))/4}(z),
\end{equation}
where $\varphi_t$ is the time-$t$ advance map of the flow, and where we have 
overloaded the notation $\cT$ by writing $\cT(z)$ to denote the period
of the point $z\in\R^2$.
For small $\lambda$, the fourth iterates of $\cF$ produce a small increment of the flow:
$$ 
\cF^4(z) = \varphi_{\lambda - \cT(z)}(z) = \phil(z) 
$$
where points of the form $\varphi_{\pm\lambda/2}(x,x)$ are identified
since they are connected by the action of $\varphi_\lambda$.

We define a first-return map for $\cF$.
For the Poincar\'{e} section $\cX$ we choose the set of points in the 
first quadrant which are closer to $\Fix{G}$ than their preimages under $\phil$, 
and at least as close as their images:
\begin{equation} \label{eq:cX}
 \cX = \{ \varphi_{\lambda\theta}(x,x) \, : \; x\geqslant 0, \; \theta\in[-1/2,1/2) \}.
\end{equation}

Since the period function $\cT(\alpha)$ of the flow is piecewise-affine on the critical intervals $\cIe$,
we expect the first return map of $\cF$ to be affine on each of the polygon classes.
We isolate the behaviour corresponding to each polygon class by defining a 
sequence of sub-domains $\cX^e\subset\cX$, for $e\in\cE$, 
whose union has full measure in $\cX$ as $\lambda\to 0$.

The polygon class associated with the critical number $e\in\cE$ meets the positive half symmetry line
in the box $B_{m,m}$, where $m=\vI$, so that the local component of the vector field $\bfw$ is $\bfw_{m,m}$.
We call a point $z\in\cX$ with $\cH(z)\in\cIe$ \defn{regular} with respect to the flow if
\begin{equation} \label{eq:regular_pts}
 \phil(z) = z + \lambda\bfw(z) = z + \lambda\bfw_{m,m}.
\end{equation}
Then the family of sets $\cX^e$ are given by:
\begin{equation} \label{eq:cXe}
 \cX^e = \{ z\in\cX \, : \; \cH(z) \in \Ie(\lambda) \} \hskip 40pt e\in\cE,
\end{equation}
where $\Ie(\lambda)\subset \cIe$ is the largest interval such that all points in $\cX^e$ are regular.
Since all points which are not regular must lie in an $O(\lambda)$ neighbourhood of the boundary of $B_{m,m}$,
it is straightforward to show that $|\Ie(\lambda)|\to|\cIe|$ as $\lambda\to 0$.

\subsection{A change of coordinates} \label{sec:ChangeOfCoordinates}

We introduce canonical coordinates for the return map on each polygonal 
class, which will facilitate the comparison between classes.
We let $\X$ denote the unit cylinder:
\begin{equation}\label{eq:bbX}
\X=\bbS^1\times\R.
\end{equation}
Then, for $e\in\cE$, we define the two-parameter family of maps:
 $$ \eta^e_\lambda: \cX^e \rightarrow \X \hskip 40pt e\in\cE, \; \lambda>0. $$
(We shall often omit the $\lambda$-dependence in the notation.)
The map $\eta^e_\lambda$ is a change of coordinates $z=(x,y)\mapsto (\theta,\rho)$, with
\begin{equation} \label{eq:rho_theta}
 \theta(z) = \frac{1}{\lambda} \, \frac{x-y}{2(2\vI+1)} \hskip 20pt \rho(z) = \frac{1}{\lambda} \, \frac{x+y-2x_0}{2(2\vI+1)},
\end{equation}
where $z_0=(x_0,x_0)\in\cX^e$ is some fixed point of the return map lying on $\Fix{G}$. 
(Below, in the proof of theorem \ref{thm:Omega_e}, we show that such a fixed point 
is guaranteed to exist for all sufficiently small $\lambda$.)
This change of coordinates is the composition of several transformations: 
a rotation through an angle $\pi/4$, which maps the symmetry line $\Fix{G}$ onto the $\rho$-axis; 
a rescaling of the plane by a factor of $1/\lambda\sqrt{2}(2\vI+1)$, which normalises the range of the coordinate $\theta$; 
a translation, which sets the preimage of the origin $(\theta,\rho)=(0,0)$ at the fixed point $z_0$.

For $z,\varphi_{\lambda t}(z)\in\cX^e$, the flow acts as:
 $$ \varphi_{\lambda t}(z) = z + \lambda t\bfw(z), $$
where $\bfw(z)$ is perpendicular to $\Fix{G}$ (see equation \eqref{eq:regular_pts}).
Thus the coordinate $\theta$ is parallel to the direction of flow, and $\rho$ is perpendicular to it:
\begin{equation*} 
 \varphi_{\lambda t}(\theta, \rho) = (\theta + t, \rho).
\end{equation*}
(We use the symbol $\varphi$ to denote the flow in both coordinate spaces.)
If we identify the interval $[-1/2,1/2)$ with the unit circle $\bbS^1$, 
then we see that the coordinate $\theta$ plays the same role as in the expression \eqref{eq:cX} for $\cX$.

In the following proposition, we show that under this change of coordinates, the unperturbed return 
map becomes a linear twist map.

\begin{theorem} \label{thm:Omega_e}
For $e\in\cE$, let $T^e$ be the map
\begin{equation} \label{eq:Omega^e}
 T^e : \mathbb{X}\rightarrow \mathbb{X}
    \hskip 40pt 
 T^e(\theta,\rho) = \left( \theta - \kappa \rho,\, \rho \right).
\end{equation}
where $\mathbb{X}$ was defined in (\ref{eq:bbX}),
\begin{equation} \label{eq:kappa}
\kappa(e) = -\frac{1}{2} (2\vI+1)^2 \cT^{\prime}(e),
\end{equation}
and $\cT$ is the period function given by (\ref{eq:cT(alpha)}).
Then for any sufficiently small $\lambda>0$, the following diagram commutes:
$$
\def\mapright#1{\smash{\mathop{\longrightarrow}\limits^{#1}}}
\def\mapdown#1{\Big\downarrow\rlap{$\vcenter{\hbox{$\scriptstyle#1$}}$}}
\begin{matrix}
\cX^e&\mapright{\cF^t}&\cX^e\cr
\mapdown{\eta}&&\mapdown{\eta}\cr
\mathbb{X}&\mapright{T^e}&\mathbb{X}\cr
\end{matrix}
$$
where $t$ is the first-return time to $\cX^e$.
\end{theorem}

\begin{proof}
For $e\in\cE$, pick $z\in\cX^e$ and let $(\theta,\rho)=\eta^e(z)$. 
Furthermore, suppose that $t$ is the return time of $z$ to $\cX^e$, and let $z^{\prime}=\cF^t(z)$.
We begin by finding an expression for the return time $t$.

If $z=\varphi_{\lambda\theta}(x,x)$ for some $x\geqslant 0$, then by the definition (\ref{eq:cF}) of $\cF$ we have
\begin{displaymath}
 \cF^k(z) = \varphi_{\lambda\theta + k(\lambda - \cT(z))/4} (x,x) \hskip 40pt k\in\Z.
\end{displaymath}
Thus, by the expression (\ref{eq:cX}) for $\cX$, the $k$th iterate of $z$ under $\cF$ lies in the set $\cX$ if and only if
\begin{displaymath}
 \theta + k \left(\frac{\lambda - \cT(z)}{4\lambda}\right) + \frac{m\cT(z)}{\lambda} \in [-1/2,1/2)
\end{displaymath}
for some $m\in\Z$. The return time $t$ is the minimal $k\in\N$ for which this inclusion holds. 
Writing $k=4l+r$ for $l\in\Z_{\geqslant 0}$ and $0\leqslant r\leqslant 3$, it is straightforward to see that 
$k$ is minimal when $r=1$, $m=l$, and $l$ satisfies
\begin{displaymath}
 \theta + l - \frac{ \cT(z)}{4\lambda} \in [-3/4,1/4),
\end{displaymath}
i.e., when
\begin{displaymath}
 l + \Bfl{\theta + \frac{3}{4} - \frac{\cT(z)}{4\lambda}} =0.
\end{displaymath}
Thus the return time is given by
\begin{equation} \label{eq:tau(z)}
  t = 4\Bceil{ \frac{\cT(z)}{4\lambda} - \theta - \frac{3}{4} } + 1,
\end{equation}
where we have used the relation $-\fl{x}=\ceil{-x}$.

Now if $z^{\prime}=\varphi_{\lambda\theta^{\prime}}(x,x)$, where $\theta^{\prime}\in[-1/2,1/2)$, 
then by construction, we have
 $$ \lambda\theta^{\prime} \equiv \lambda\theta + t\left(\frac{\lambda - \cT(z)}{4}\right) \mod{\cT(z)}, $$
where we write $a \equiv b ~ (\mathrm{mod} ~ c)$ for real $c$ to denote that $(a-b)\in c\,\Z$. Thus it follows from the formula (\ref{eq:tau(z)}) for the return time that
\begin{align}
 \lambda\theta^{\prime} 
  &\equiv \lambda\theta + \lambda \Bceil{ \frac{\cT(z)}{4\lambda}-\frac{3}{4}-\theta } + \frac{\lambda - \cT(z)}{4} \mod{\cT(z)} \nonumber \\
  &\equiv -\lambda \Bfl{ \theta+\frac{3}{4}-\frac{\cT(z)}{4\lambda} } + \lambda\left( \theta+\frac{3}{4}-\frac{\cT(z)}{4\lambda}\right) - \frac{\lambda}{2} \mod{\cT(z)} \nonumber \\
  &\equiv \lambda \left\{ \theta+\frac{3}{4}-\frac{\cT(z)}{4\lambda} \right\} - \frac{\lambda}{2}  \mod{\cT(z)}, \label{eq:theta_prime}
\end{align}
where $\{x\}$ denotes the fractional part of $x$.
Since the expression on the right hand side of \eqref{eq:theta_prime} lies in the interval $[-\lambda/2,\lambda/2)$,
equation \eqref{eq:theta_prime} must in fact be an equality.
Thus, dividing through by $\lambda$, we obtain
$$ 
\theta^{\prime} \equiv \theta+\frac{1}{4}-\frac{\cT(z)}{4\lambda} \mod{1}. 
$$

We wish to apply the change of coordinates $\eta^e$.
The fixed point $z_0$ used in the definition of $\eta^e$ may be any $z\in\cX^e$ satisfying
$\theta=\theta^{\prime}=0$, i.e., $z_0\in\Fix{G}$ and
\begin{equation} \label{eq:z_0}
 0 \equiv \frac{1}{4}-\frac{\cT(z_0)}{4\lambda} \mod{1}.
\end{equation}
Provided that $\cT^{\prime}(e)$ is non-zero, we can always choose such a $z_0$ as long as
$\lambda$ is sufficiently small.
Then, in the new coordinates, we have
 $$ \eta^e (z^{\prime}) = \left( \theta + \frac{1}{4} - \frac{\cT(z)}{4\lambda} , \rho \right), $$
where we have used the fact that the coordinate $\theta$ is periodic.
 
Now let $z=(x,y)$ and $z_0=(x_0,x_0)$. 
Since the Hamiltonian function $\cH$ is affine on $\cX^e$, expanding $\cH$ about $z_0$ gives
\begin{align}
 \cH(z) &= \cH(z_0) + (x+y-2x_0)(2\vI+1) \nonumber \\
 &= \cH(z_0) + 2\lambda(2\vI+1)^2\rho. \label{eq:cP(z)_rho_expansion}
\end{align}
Similarly, since $\cT$ is affine on the interval $\cIe$:
 $$ \cT(z) = \cT(z_0) + 2\lambda(2\vI+1)^2\rho\cT^{\prime}(e). $$
Finally, using the property \eqref{eq:z_0} of the fixed point $z_0$, we have
\begin{align*}
 \eta^e (z^{\prime}) &= \left( \theta + \frac{1}{4} - \frac{\cT(z_0) + 2\lambda(2\vI+1)^2\rho\cT^{\prime}(e)}{4\lambda} , \rho \right) \\
 &= \left( \theta - \frac{1}{2} (2\vI+1)^2\rho\cT^{\prime}(e) , \rho \right),
\end{align*}
which completes the proof.
\end{proof}

Note that the action of the conjugate map $T^e$ is independent of the parameter $\lambda$, 
which appears only in the domain $\eta^e_\lambda(\cX^e)$ in which the conjugate map is valid. 
Using the definition \eqref{eq:cXe} of the domain $\cX^e$ and the above expansion 
\eqref{eq:cP(z)_rho_expansion} of $\cH$ about the fixed point $z_0$, 
we see that this domain is given by
\begin{displaymath}
 \eta^e_\lambda(\cX^e) = \{ (\theta,\rho)\in\X \, : \; \cH(z_0) + 2\lambda(2\vI+1)^2\rho \in \Ie(\lambda) \}.
\end{displaymath}
The range of values of $\rho$ in the domain grows like
\begin{displaymath}
 \frac{|\Ie(\lambda)|}{2\lambda(2\vI+1)^2} = \frac{|\cIe|}{2\lambda(2\vI+1)^2} + O(1)
\end{displaymath}
as $\lambda\rightarrow 0$, and in the integrable limit, we think of the conjugate map as valid on the whole of $\X$.

For $e\in\cE$, let us consider the twist $\kappa(e)$ of the map $T^e$, given in (\ref{eq:kappa}).
Each $T^e$ is reversible, and can be written as the composition of the involutions $G$ and $H^e$, where
\begin{equation*}
 G(\theta,\rho) = (-\theta,\rho) \hskip 40pt H^e(\theta,\rho) = \left( -\theta +\rho\kappa(e),\, \rho \right).
\end{equation*}
We use the same symbol $G$ to denote the reversing symmetry on the plane 
(equation \eqref{eq:Dihedral}) and on the cylinder.
The fixed spaces of these involutions are given by 
\begin{align}
 \Fix{G} &= \{ (\theta, \rho) \in \X  \, : \; \theta\in\{-1/2,0\} \}, \nonumber \\ 
 \Fix{H^e} &= \{ (\theta, \rho) \in \X \, : \; \theta=-\frac{1}{2} \rho\kappa(e) \}. \label{eq:Fix(cH)} 
\end{align}

\begin{figure}[t]
  \centering
  \includegraphics[scale=1]{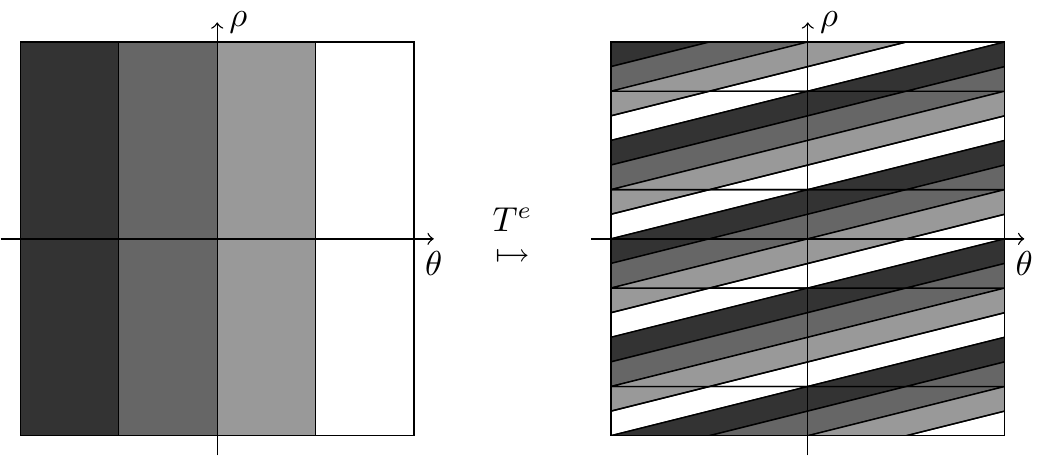} 
  \caption{A schematic representation of the action of the twist map 
$T^e$ on the space $\X=\bbS^1\times\R$}
\end{figure}

The map $T^e$ is also reversible with respect to the reflection $(\theta,\rho)\mapsto(\theta,-\rho)$,
and equivariant under the group generated by the translation
\begin{equation} \label{eq:rho_bar}
 \rho \mapsto \rho + {\rho}^* \hskip 40pt {\rho}^* = \frac{1}{\kappa(e)}=\frac{-2}{(2\vI+1)^2\cT^{\prime}(e)}.
\end{equation}
Each circle $\rho=$ const. is invariant under $T^e$, 
and motion restricted to this circle is a rotation with rotation number $\rho/{\rho}^*$ (mod $1$).

\subsection{The integrable return map at infinity}

Now we are in a position to study the dynamics of the sequence of maps $T^e$ in the limit $e\rightarrow\infty$. 
As in section \ref{sec:Nonlinearity}, where we studied the nonlinearity of the flow, 
we turn our attention to the behaviour of the period function $\cT$.

As $\alpha\to\infty$, the period $\cT(\alpha)$ converges to $\pi$,
which is the period of the rotation with Hamiltonian $\cQ(x,y)=x^2 + y^2$.
Furthermore, the period undergoes damped oscillations about its limiting value \cite{ReeveBlack}.
If we write
\begin{equation} \label{eq:b}
 \alpha(n,b) = (n+b)^2 \hskip 40pt n\in\N, \; b\in[0,1),
\end{equation}
so that $n$ is the integer part of $\sqrt{\alpha(n,b)}$ and $b$ is the fractional part,
then one can show that the following limiting function exists:
 $$ \lim_{n\to\infty} \left( n^{3/2} (\cT(\alpha(n,b))-\pi) \right), $$
and is described qualitatively by the following function of the unit interval
\begin{equation} \label{eq:T_asymptotics}
 b \mapsto \frac{1}{3}(2b+1)^{3/2}-\sqrt{2b}
\end{equation}
(see figure \ref{fig:cTrho}).
The key feature of this limiting function is the singularity in its derivative at $b=0$.

\begin{figure}[t]
        \centering
\begin{minipage}{7cm}
        \includegraphics[scale=0.30]{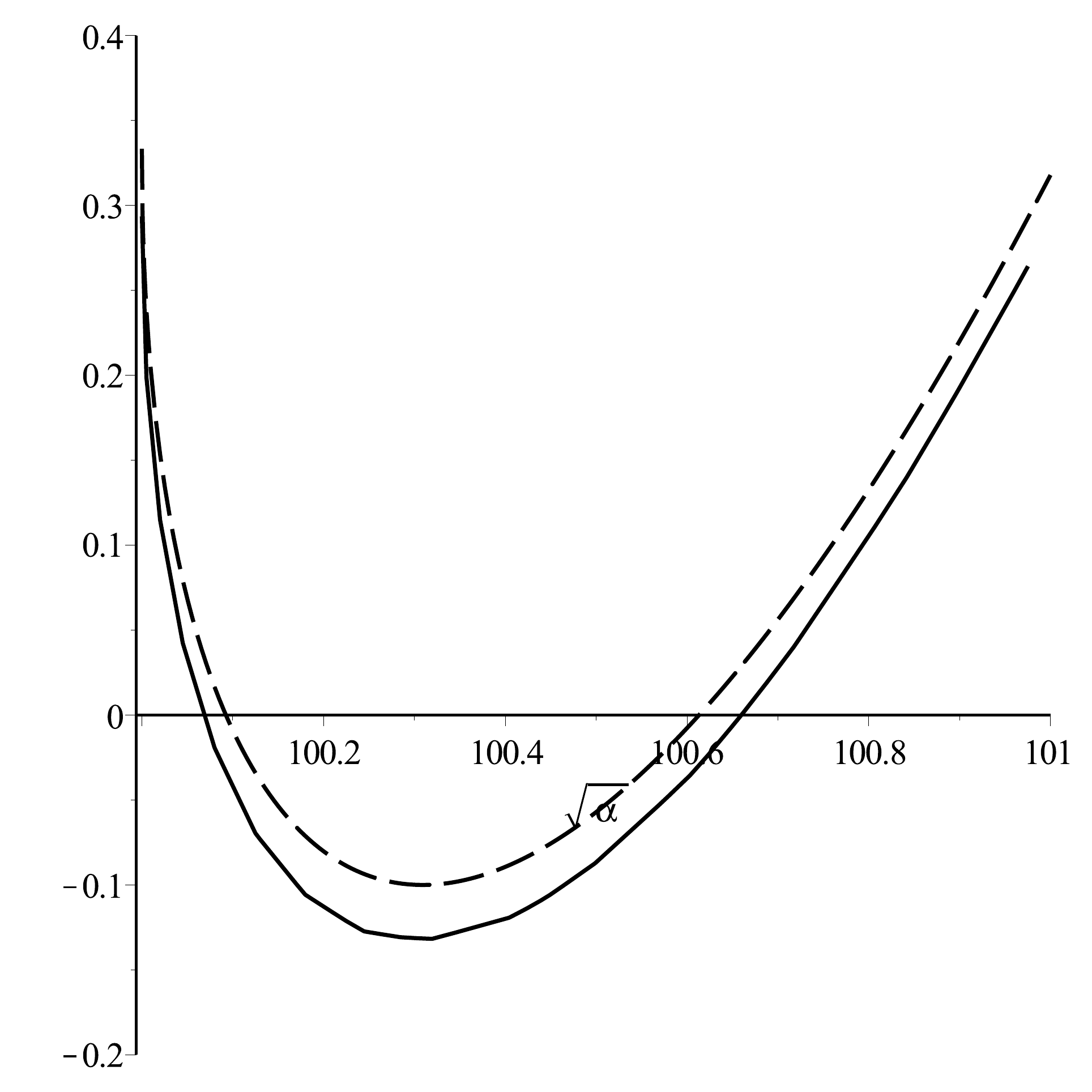} 
\end{minipage}
\begin{minipage}{7cm}
	  \includegraphics[scale=0.30]{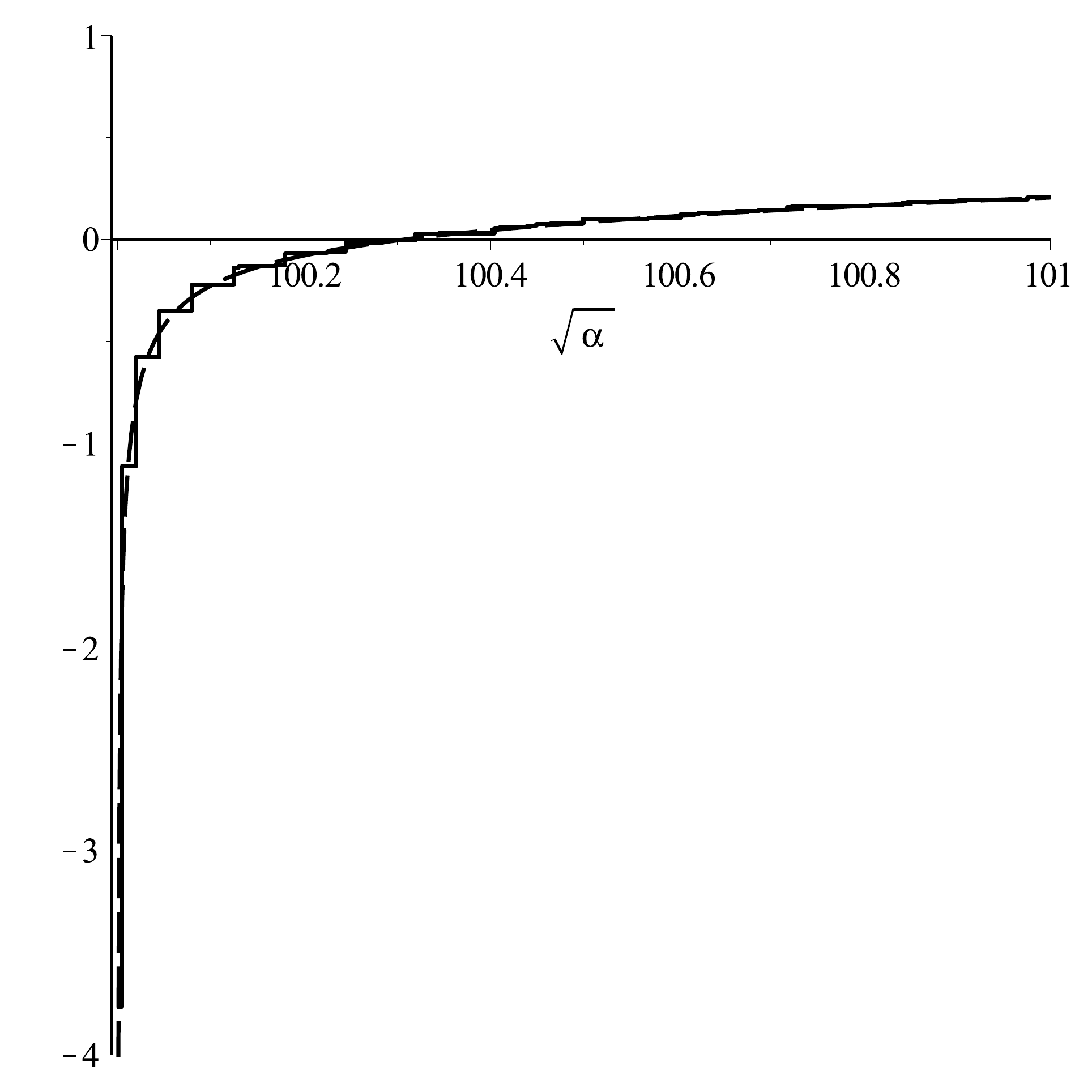} 
\end{minipage}
        \caption{Left: the function $\flsq{\alpha}^{3/2} \left(\cT(\alpha)-\pi\right)/4$ 
(solid line) against $\sqrt{\alpha}$ for $\sqrt{\alpha}\in[100,101)$, i.e., 
for $n=\flsq{\alpha}=100$ and $b=\{\sqrt{\alpha}\}\in[0,1)$. 
The dotted line shows the function $\frac{1}{3}(2b+1)^{3/2} - \sqrt{2b}$.
Right: the function $\frac{1}{2}(2\flsq{\alpha/2}+1)^2 \cT^{\prime}(\alpha)$ of equation \eqref{eq:rho_bar}
       (solid line). The dotted lines show the approximate limiting function \eqref{eq:ApproximateDerivative}.
        \label{fig:cTrho}}
\end{figure}

With $\alpha(n,b)$ as in \eqref{eq:b}, we let $e(n,b)$
be the critical number associated with $\alpha$:
\begin{equation} \label{eq:e(n,b)}
 (n+b)^2 \in \{e(n,b)\}\cup \cI^{e(n,b)} \hskip 40pt n\in\N, \; b\in[0,1).
\end{equation}
Then the behaviour of the sequence of twists $\kappa(e(n,b))$ as $n\to\infty$ is described by the
following theorem.

\begin{theorem} \label{thm:Tprime_asymptotics}
Let $b\in[0,1)$ and let $e(n,b)$ be given as in \eqref{eq:e(n,b)}.
Then as $n\rightarrow \infty$:
\begin{equation} \label{eq:Tprime_asymptotics}
 \kappa(e(n,b)) \rightarrow  
\left\{ \begin{array}{ll} 4 \quad & x=0, \\ 0 \quad & \mathrm{otherwise.}\end{array} \right..
\end{equation}
\end{theorem}

\begin{proof}
Take $\alpha\geqslant 0$ and let $e\in\cE$ be such that $\alpha\in\cIe$.
From the formula \eqref{eq:Tprime(alpha)} for the derivative of the period function, we have
\begin{equation} \label{eq:Tprime(alpha)_II}
 \frac{1}{2} (2\vI+1)^2 \cT^{\prime}(\alpha) = 2 -8(2\vI+1)^2 \sum_{k=\vI+1}^{\vk} \frac{1}{(4k^2-1)(2\flsq{\alpha - k^2}+1)}.
\end{equation}
(Recall that if $\alpha\in\cIe$ for some $e\in\cE$, the expression in the denominator of the summand 
satisfies $\flsq{\alpha - k^2} = \flsq{e - k^2}$.)

We consider the summand in \eqref{eq:Tprime(alpha)_II}. For $k$ in the range $\vI+1\leqslant k\leqslant \vk-1$, we have
\begin{equation*}
 \frac{1}{(4k^2-1)(2\fl{\sqrt{\alpha - k^2}}+1)} =  \frac{1}{8k^2\sqrt{\alpha - k^2}} \, \left( 1 - \frac{1}{4k^2} \right)^{-1} \left( 1 - \frac{\{\sqrt{\alpha - k^2}\}-1/2}{\sqrt{\alpha - k^2}} \right)^{-1}.
\end{equation*}
Here the square root is bounded below by
 $$ \sqrt{\alpha - k^2} \geqslant \sqrt{\vk^2 - (\vk-1)^2} = \sqrt{2\vk -1} > \alpha^{1/4}, $$
whereas $k^2 >\alpha/2$. Hence as $\alpha\rightarrow\infty$, we have
\begin{equation} \label{eq:Tprime_bound1}
 \frac{1}{(4k^2-1)(2\fl{\sqrt{\alpha - k^2}}+1)} = \frac{1}{8k^2\sqrt{\alpha - k^2}} \left(1 + O\left(\frac{1}{\alpha^{1/4}}\right) \right).
\end{equation}

Summing over $k$, we can approximate all but the $k=\vk$ term sum with an integral:
\begin{align*}
 \sum_{k=\vI+1}^{\vk-1} \frac{1}{k^2\sqrt{\alpha - k^2}} 
 &= \int_{\vI+1/2}^{\vk-1/2} \frac{1}{x^2\sqrt{\alpha - x^2}} \left(1 + O\left(\frac{1}{\alpha^{1/4}}\right) \right) \, dx  \\
 &= \frac{1}{\alpha} \, \Big[ \tan{\theta} \Big]_{\theta_2}^{\theta_1} \left(1 + O\left(\frac{1}{\alpha^{1/4}}\right) \right)
\end{align*}
where we have used the substitution $x=\sqrt{\alpha}\cos{\theta}$, 
and the limits $\theta_1$ and $\theta_2$ satisfy:
\begin{align*}
 \cos{\theta_1} = \frac{\vI+1/2}{\sqrt{\alpha}} \hskip 40pt 
 \cos{\theta_2} = \frac{\vk-1/2}{\sqrt{\alpha}}. 
\end{align*}
Using Taylor's Theorem, we find that $\tan{\theta_1}$ and $\tan{\theta_2}$ 
are given by
\begin{align*}
 \tan{\theta_1} = 1 + O\left(\frac{1}{\sqrt{\alpha}}\right) \hskip 40pt 
 \tan{\theta_2} = O\left(\frac{1}{\alpha^{1/4}}\right), 
\end{align*}
so that
\begin{equation} \label{eq:Tprime_bound2}
 \sum_{k=\vI+1}^{\vk-1} \frac{1}{k^2\sqrt{\alpha - k^2}} 
 = \frac{1}{\alpha} \left(1 + O\left(\frac{1}{\alpha^{1/4}}\right) \right).
\end{equation}
Applying \eqref{eq:Tprime_bound1} and \eqref{eq:Tprime_bound2} to the sum in \eqref{eq:Tprime(alpha)_II}, we obtain
\begin{align*}
 & 8(2\vI+1)^2 \sum_{k=\vI+1}^{\vk-1} \frac{1}{(4k^2-1)(2\flsq{\alpha - k^2}+1)} \\
 &= (2\vI+1)^2 \sum_{k=\vI+1}^{\vk-1} \frac{1}{k^2\sqrt{\alpha - k^2}} \left(1 + O\left(\frac{1}{\alpha^{1/4}}\right) \right) \\
 &= \frac{1}{\alpha} (2\vI+1)^2 \left(1 + O\left(\frac{1}{\alpha^{1/4}}\right) \right) \\
 &= 2 + O\left(\frac{1}{\alpha^{1/4}} \right)
\end{align*}
as $\alpha\rightarrow\infty$. Thus substituting this back into \eqref{eq:Tprime(alpha)_II}, we see that only the $k=\vk$ term remains.
If $\alpha=(n+b)^2$, then $\vk=n$, and the remaining term is given by:
\begin{align} 
 \frac{1}{2} (2\vI+1)^2 \cT^{\prime}(\alpha) 
 &= \frac{-8(2\vI+1)^2}{(4n^2-1)\left(2\flsq{\alpha-n^2}+1\right)} + O\left(\frac{1}{\alpha^{1/4}} \right) \nonumber \\
 &= \frac{-4}{2\flsq{2nb+b^2}+1} \left(1 + O\left(\frac{1}{\sqrt{\alpha}}\right) \right) + O\left(\frac{1}{\alpha^{1/4}} \right). \label{eq:Tprime_final_term}
\end{align}
If $b>0$, then
\begin{equation*}
 \frac{1}{2\flsq{2nb+b^2}+1} \rightarrow 0
\end{equation*}
as $n\rightarrow\infty$. 
Thus the remaining term in \eqref{eq:Tprime_final_term} goes to zero and $(2\vI+1)^2 \cT^{\prime}(\alpha)$ vanishes in the limit.
However, if $b=0$, then we have
 $$ \frac{1}{2} (2\vI+1)^2 \cT^{\prime}(\alpha) = -4 + O\left(\frac{1}{\alpha^{1/4}} \right). $$
\end{proof}

The convergence \eqref{eq:Tprime_asymptotics} is not uniform: 
if $b$ is close to zero, the convergence can be made arbitrarily slow.
For a plot of $(2\vI+1)^2 \cT^{\prime}(\alpha)/2$, see figure \ref{fig:cTrho}.

The two regimes of behaviour for the sequence of return maps $T^e$ follow directly.

\begin{corollary} \label{corollary:Omega_asymptotics}
Let $b\in[0,1)$ and let $e(n,b)$ be given as in \eqref{eq:e(n,b)}.
If $b=0$, then the sequence $e(n,b)$ is simply the sequence of perfect squares,
and as $n\rightarrow\infty$, the sequence of functions $T^{e(n,b)}$ 
converges pointwise to a limiting function $T^{\infty}$ given by
\begin{equation*}
 T^{\infty}: \X \rightarrow \X
 \hskip 40pt 
 T^{\infty}(\theta,\rho) = \left( \theta + 4\rho, \rho \right).
\end{equation*}
If $b>0$, then the sequence $T^{e(n,b)}$ converges pointwise to the identity.
\end{corollary}

The (reversing) symmetries of $T^{e(n,b)}$ also converge as $n\rightarrow\infty$, 
leading to an asymptotic version of $H^e$. 
Here we note in particular the translation invariance \eqref{eq:rho_bar} of the sequence $T^{e(n,b)}$, whose magnitude satisfies
\begin{equation} \label{eq:rho_bar_asymptotics}
 |{\rho}^*| \rightarrow \left\{ \begin{array}{ll} 1/4 & \quad b=0 \\ \infty & \quad b>0 \end{array} \right.
\end{equation}
as $n\rightarrow\infty$.

\section{The perturbed dynamics} \label{sec:PerturbedDynamics}

Let us return to the perturbed dynamics, i.e., to the rescaled discretised rotation $\F$.
We saw that under a suitable change of coordinates, the unperturbed return map corresponds 
to a linear twist map $T^e$ on the cylinder, whose dependence on the parameter is singular 
in the limit $e\to\infty$. There is a corresponding perturbed twist map $\Phi$ of a 
discretised cylinder, which captures the recurrent behaviour of $\F$.
As $e\to\infty$ through a sub-sequence of $(e_i)$ (cf.~(\ref{eq:cE})), we get a sequence of 
discrete twist maps with vanishing discretisation length. The resulting dynamics will depend 
on the particular sub-sequence chosen.
In the regime where the twist vanishes in the limit, the round-off fluctuations result in 
intricate resonance structures, reminiscent of the island chains of near-integrable
Hamiltonian systems. Conversely, in the regime where the twist persists, there are no 
discernible phase space features and a form of `pseudochaos' emerges.

Throughout what follows, we adopt an informal approach, focussing on qualitative 
observations and numerical evidence.

\subsection{The perturbation mechanism} \label{sec:PerturbationMechanism}

The correspondence between the integrable flow described in the previous 
sections and the behaviour of the rescaled lattice map $\F$ in the limit 
$\lambda\to 0$ is illustrated by the following proposition,
which is a simple consequence of \cite[Proposition 1]{ReeveBlackVivaldi}.
It says that, within a bounded region of the plane, the lattice points
over which the time-advance map and the round-off dynamics agree have
full density as $\lambda\to 0$.

\begin{proposition} \label{prop:mu_1}
 For $r>0$, we define the set
\begin{equation}\label{eq:A}
 A(r,\lambda) = \{ z\in\lZ \, : \; \| z \|_{\infty} < r \}
\end{equation}
(with $\| (u,v) \|_{\infty} = \max (|u|,|v|)$).
Then
 \begin{displaymath}
  \lim_{\lambda\to 0} 
   \frac{ \# \{z\in A(r,\lambda) \, : \; \phil(z) = \F^4(z) \} }{\# A(r,\lambda)} =1.
 \end{displaymath}
\end{proposition}

We isolate the lattice points $z\in\lZ$ where the fourth iterates of $\F$ deviate 
from the time-advance map $\phil$ by defining the set $\Lambda$ of \defn{transition points}. 
We say that a point $z\in\lZ$ is a transition point if $z$ and its image under $F_\lambda^4$
do not belong to the same box.
Then
\begin{equation} \label{eq:Lambda}
 \Lambda = \bigcup_{m,n\in\Z} \Lambda_{m,n},
\end{equation}
where
\begin{displaymath}
 \Lambda_{m,n} =  \F^{-4}(B_{m,n}\cap\lZ)\setminus B_{m,n}.
\end{displaymath}
For sufficiently small $\lambda$, 
all points $z\not=(0,0)$ where $\F^4(z)\neq z+\lambda\bfw(z)$ 
are transition points \cite[Lemma 4]{ReeveBlackVivaldi}. 
Hence $\F^4$ is equal to the time-$\lambda$ advance of the 
flow everywhere except at the transition points.

\begin{proposition} \label{prop:Lambda}
Let $A(r,\lambda)$ be as in equation \eqref{eq:A}.
Then for all $r>0$ there exists $\lambda^*>0$ such that, for all $\lambda<\lambda^*$ 
and $z\in A(r,\lambda)$, we have
\begin{displaymath}
 z \notin \Lambda \cup \{(0,0)\} \quad \Rightarrow \quad \F^4(z)= \phil(z).
\end{displaymath}
\end{proposition}

\begin{figure}[t]
        \centering
        \includegraphics[scale=0.9]{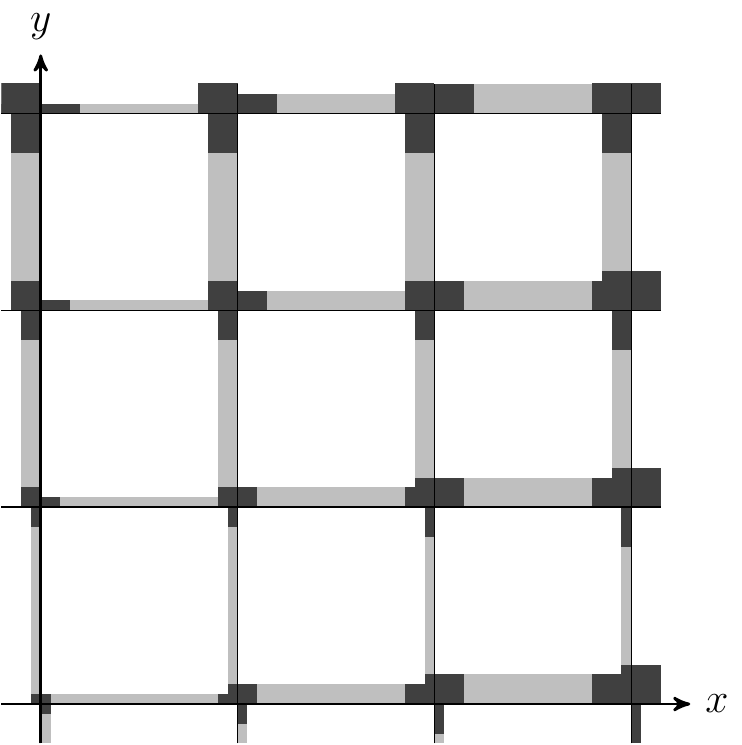}
        \caption{Structure of phase space. The boxes $B_{m,n}$, bounded by the
         set $\Delta$, include regular domains (white), where the motion is integrable. 
         The light grey regions comprise the set $\Lambda$ of transition points, given in \eqref{eq:Lambda}. 
         The dark grey regions near the integer lattice $\Z^2$ is the set $\Sigma$. The orbits
         that intersect it are the perturbed analogue of the critical polygons, and are excluded 
         from our analysis. }
        \label{fig:LambdaSigma_plot}
\end{figure}

The transition points play the role of the vertices for the perturbed orbits,
and are the basis of the strip-map construction for the map $\F$:
for small $\lambda$, the set of transition points consists of thin 
strips of lattice points arranged along the lines $\Delta$ (see figure \ref{fig:LambdaSigma_plot}).
The \textbf{strip map} $\Psi$ is the transit map to $\Lambda$,
which is defined as follows:
\begin{equation} \label{eq:Psi}
  \Psi : \lZ \rightarrow \Lambda
 \hskip 40pt
  \Psi(z) = \F^{4t(z)}(z),
\end{equation}
where the transit time $t$ to $\Lambda$ is well-defined for all points excluding the origin.
(Since the origin plays no role in the present construction, to simplify 
notation we shall write $\lZ$ for $\lZ\setminus\{(0,0)\}$, where appropriate.)
The strip map moves points parallel to the relevant component of the piecewise-constant
vector field $\bfw$, with a possible additional small perturbation.
If $z\in\Lambda_{m,n}$:
\begin{equation}\label{eq:Psi^jp1}
 \Psi(z) = z + \lambda t(z)\bfw_{m,n} + \lambda(\bfv(z)-\bfw(z)).
\end{equation}

We shall exclude the regions of the phase space where multiple strips overlap,
which form a set $\Sigma\subset\lZ$, clustered around the lattice points of 
$\mathbb{Z}^2$ \cite[Section 5]{ReeveBlackVivaldi}.
Thus we exclude the orbits that intersect this set ---the so called \defn{critical orbits}.

In \cite[section 3]{ReeveBlackVivaldi}, we defined a Poincar\'e return map 
$\Phi$ on a neighbourhood $X$ of the positive half of the symmetry line $\Fix{G}$.
For time-scales corresponding to a first return to the domain $X$, every 
integrable orbit has a scaled return orbit that shadows it \cite[theorem 5]{ReeveBlackVivaldi}.
The orbits which are not critical form a sequence of subsets $\Xe$ of $X$, $e\in\cE$,
which correspond to the polygon classes of the flow.
The set $\Xe$ is the intersection of the lattice $\lZ$ with a thin rectangle lying along 
the symmetry line $\Fix{G}$:
\begin{equation}\label{eq:Xe}
 \Xe = \{ \lambda(x,y)\in\lZ \, : \;  -(2\vI+1) \leqslant x-y < 2\vI+1, \; \cP(\lambda x, \lambda y)\in\Ie(\lambda) \},
\end{equation}
where, just as in equation (\ref{eq:cXe}), the set $\Ie(\lambda)$ is a subset of 
$\cIe$ whose length tends to that of $\cIe$ 
as $\lambda\to0$ \cite[Proposition 6]{ReeveBlackVivaldi}:
\begin{displaymath}
 \lim_{\lambda\rightarrow 0}\frac{ |\Ie(\lambda)| }{|\cIe|} = 1.
\end{displaymath}
From this it follows that the sets $\Xe$ have full density in $X$ as $\lambda\to 0$.
We identify opposite sides of the set $\Xe$ which are connected by the local vector
field so that the dynamics takes place on a lattice on a cylinder.

For orbits which are not critical, the return map $\Phi$ can be constructed from repeated 
iterations of the strip map.  
For $e\in\cE$ and $z\in\Xe$:
\begin{equation} \label{eq:Phi_Psi}
\Phi(z) \equiv (\Psi^{2\vk+2}\circ F_{\lambda})(z) \mod{\lambda\bfw(z)},
\end{equation}
where the modulus reflects the periodicity of the transversal coordinate.
(Here we write $(\mathrm{mod}\,\mathbf{a})$ for some vector $\mathbf{a}$ to denote 
congruence modulo the one-dimensional module generated by $\mathbf{a}$.)
Relative to the integrable flow, each iterate of the strip map $\Psi$ produces an 
elementary perturbation of the type described in the introduction (figure \ref{fig:Scattering}),
and the number of such perturbations which contribute to the return map $\Phi$ diverges as $e\to\infty$.

\subsection{The perturbed return map} \label{sec:PerturbedReturnMap}

We turn our attention to the reversing symmetry group of the return map 
$\Phi$, which is inherited from the reversing symmetry $G$ of the lattice 
map $\F$ (see equation \eqref{eq:Dihedral}). 
In the following proposition we describe the form of this reversing symmetry 
on the domains $\Xe$ (equation \eqref{eq:Xe}).
Then we apply the change of coordinates $\eta^e$ from the previous section
to the domain $\Xe$ and compare the properties of the resulting discrete 
perturbed twist map to that of the unperturbed twist map $T^e$.

\begin{proposition} \label{prop:Ge}
 For $e\in\cE$, let $G^e$ be the involution of $\Xe$ given by 
 \begin{equation} \label{eq:Ge}
  G^e(x,y) = \left\{ 
      \begin{array}{ll} (y,x) \quad & |x-y| < \lambda(2\vI+1) \\
		      (x,y) \quad & x-y = -\lambda(2\vI+1).
      \end{array} \right.
 \end{equation}
Then for all sufficiently small $\lambda$, $G^e$ is a reversing symmetry of 
$\Phi$ on $\Xe$ in the following sense:
\begin{equation*}
 \forall z,\Phi(z)\in\Xe:
 \hskip 20pt 
 \Phi^{-1}(z) = (G^e \circ \Phi \circ G^e)(z).
\end{equation*}
\end{proposition}

\begin{proof}
Recall that $\Xe\subset B_{m,m}\setminus\Lambda$, where $m=\vI$, so that if $z\in\Xe$, then
 $$ \F^4(z) = z + \lambda\bfw(z) = z + \lambda\bfw_{m,m}. $$
Furthermore, if $z$ lies on the line $x-y=-\lambda(2\vI+1)$, 
then by the definition \eqref{eq:w_mn} of the auxiliary vector field, we have
 $$ G(z) = z + \lambda\bfw_{m,m}. $$
Combining the above, we have the following relationship between $G^e$ and $G$:
\begin{equation} \label{eq:Ge_G}
 G^e(x,y) = \left\{ 
      \begin{array}{ll} G(x,y) \quad & |x-y| < \lambda(2\vI+1), \\
		      (\F^{-4}\circ G)(x,y) \quad & x-y = -\lambda(2\vI+1).
      \end{array} \right.
\end{equation}
 
If $\tau=\tau(z)$ is the return time of $z$, then by construction
  $$ \Phi(z) = \F^{\tau}(z), $$
and the reversibility of $\F$ with respect to $G$ gives us that
\begin{equation} \label{eq:Phi_reversibility}
 (G \circ \Phi)(z) = (\F^{-\tau} \circ G)(z).
\end{equation}

Suppose now that neither $z$ nor $\Phi(z)$ lies on the line $x-y=-\lambda(2\vI+1)$.
Then combining \eqref{eq:Ge_G} and \eqref{eq:Phi_reversibility} gives
 $$ (G^e \circ \Phi)(z) = (\F^{-\tau} \circ G^e)(z). $$
Furthermore, this point lies in $\Xe$, so that
 $$ (G^e \circ \Phi)(z) = (\Phi^{-1} \circ G^e)(z), $$
as required.

Suppose now that $\Phi(z)$ lies on the line $x-y=-\lambda(2\vI+1)$ but $z$ does not.
In this case, combining \eqref{eq:Ge_G} and \eqref{eq:Phi_reversibility} gives
 $$ (G^e \circ \Phi)(z) = (\F^{-4} \circ G \circ \Phi)(z) = (\F^{-(\tau+4)} \circ G)(z) = (\Phi^{-1} \circ G^e)(z). $$
The other cases proceed similarly.
\end{proof}

\medskip
Note that the reversing symmetry $G^e$ is simply an adaptation of the original
reversing symmetry $G$ of $\F$.

Recall the map $T^e$ (equation \eqref{eq:T^e}), 
which corresponds to the unperturbed return map under the change of coordinates 
$\eta^e(\lambda):(x,y)\mapsto(\theta,\rho)$ of equation \eqref{eq:rho_theta}. 
In the integrable limit $\lambda\to 0$, the image of the return domain 
$\cX^e$ under $\eta^e$ approaches the unit cylinder $\X$.
The domain $\Xe$ of the perturbed return map $\Phi$ is a subset of $\cX^e$, 
thus we can also apply the change of coordinates $\eta^e$ to $\Xe$.
Using the definition \eqref{eq:Xe} of $\Xe$, we see that its image under $\eta^e$ 
approaches the following set\footnote{Recall that in the definition of $\eta^e$, 
the pre-image of the origin is some fixed point $z_0$ of the unperturbed dynamics. 
In the discrete case we round $z_0$ onto the lattice. 
The choice of the fixed point $z_0$ will have no bearing on the qualitative or statistical properties of the dynamics.}
(see figure \ref{fig:Lattice2}):
\begin{equation} \label{eq:rho_theta_lattice}
 \left\{ \frac{1}{2(2\vI+1)} \, (i,i+2j) \,: \; i,j\in\Z, \; -(2\vI+1)\leqslant i < 2\vI+1 \right\} \subset \X.
\end{equation}
(The image of the point $\lambda(x,y)$ corresponds to $i=x-y$ and $j=y-x_0$.)
This set is a rotated square lattice in the unit cylinder, with lattice spacing $1/\sqrt{2}(2\vI+1)$.
Thus $e\to\infty$ corresponds to the continuum limit. 
We think of the dynamics of $\Phi$ as taking place in the $(\theta,\rho)$ coordinates, 
and of $\Phi$ as a discrete version of $T^e$.

\begin{figure}[t]
  \centering
  \includegraphics[scale=0.8]{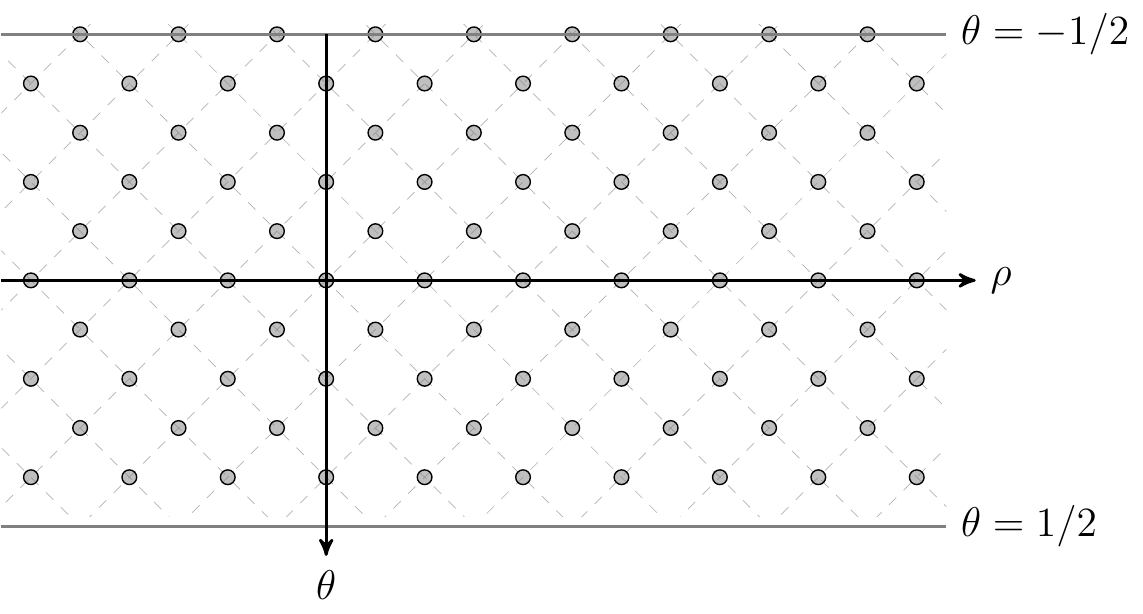} 
  \caption{The image of the domain $\Xe$ under the change of coordinates 
 $\eta^e_\lambda: \lZ\to \mathbb{X}$. 
 There are $2\vI+1$ lattice points per unit length in each of the coordinate directions.}
  \label{fig:Lattice2}
\end{figure}

The map $T^e$ is a linear twist map, with characteristic length scale $\brho$ in the 
$\rho$-direction (see equation \eqref{eq:rho_bar}).
The map $\Phi$ can be considered as a perturbation of $T^e$ which originates from space 
discretisation. As $e\to\infty$, the fluctuations caused by the perturbation are small 
relative to the width of the cylinder, do not have any obvious structure, and affect 
the dynamics in both the $\rho$ and $\theta$ directions.

\subsection{Qualitative description and phase plots}

Recall the sequence $e(n,b)$ of critical numbers defined in \eqref{eq:e(n,b)}.
By corollary \ref{corollary:Omega_asymptotics}, if $b\neq0$, then the sequence 
of maps $T^{e(n,b)}$ converge non-uniformly to the identity as $n\to\infty$,
whereas if $b=0$, then $T^{e(n,b)}$ converges to the map $(\theta,\rho)\mapsto(\theta+4\rho,\rho)$.
Accordingly, the characteristic length scale $\brho$ of the twist dynamics is also singular in the limit, 
with $|\brho|\to\infty$ or $\brho\to1/4$, for the $b\neq 0$ and $b=0$ cases, respectively 
(see equation \eqref{eq:rho_bar_asymptotics}).

We define the \defn{rotation number} $\nu$ of a point on the cylinder
to be its rotation number under the twist map $T^e$, i.e.,
 $$ \nu(\theta,\rho) = \frac{\rho}{\brho} \mod{1} \hskip 40pt (\theta,\rho)\in\X. $$
Similarly for a point $z\in\Xe$, we write $\nu(z)$ to denote the rotation number 
of the corresponding point $\eta^e(z)$ on the cylinder.
Subsets of the cylinder of the form 
\begin{equation} \label{eq:fundamental_domain}
 \left\{ (\theta,\rho) \,: \; \nu(\theta,\rho) -m \in [-1/2,1/2) \right\}\subset \X \hskip 40pt m\in\Z 
\end{equation}
are referred to as \defn{fundamental domains} of the dynamics.
The number $N$ of lattice points of $\eta^e(\Xe)$ per fundamental domain varies like
\begin{equation} \label{eq:pts_per_fundamental_domain}
 N \sim 2(2\vI+1)^2|\brho|
\end{equation}
as $e\rightarrow\infty$.
For a given value of $e$, we expect the dynamics of $\Phi$ to be qualitatively the same
in each fundamental domain, but to vary locally according to the rotation number.
Consequently, in order to observe the global behaviour of $\Phi$, 
we need to sample whole fundamental domains.
However, this is made difficult by the divergence of $\brho$ for $b\neq 0$.

To better understand the divergence of $\brho$, we consider the limiting form 
\eqref{eq:T_asymptotics} of the period function $\cT(\alpha)$.
By differentiating the function (\ref{eq:T_asymptotics}) (see figure \ref{fig:cTrho}) 
we expect, for large $e$, the piecewise-constant function $\cT^{\prime}(\alpha)$ to 
behave approximately as
\begin{equation} \label{eq:ApproximateDerivative}
\frac{1}{2}(2\flsq{\alpha/2}+1)^2 \cT^{\prime}(\alpha) \approx \frac{2}{\sqrt{n}} \left(\sqrt{2b+1}-1/\sqrt{2b}\right) 
    \hskip 40pt \alpha=(n+b)^2,\quad b\neq 0.
\end{equation}
By \eqref{eq:rho_bar}, this leads us to expect that 
\begin{equation}\label{eq:rho_barApprox}
\brho(\alpha) \approx -\frac{\sqrt{n}}{2} \left(\sqrt{2b+1}-1/\sqrt{2b}\right)^{-1} \hskip 40pt b\neq 0. 
\end{equation}
Figure \ref{fig:rho_barApprox} shows that this rough analysis is valid, 
although the relationship between the two functions is far from uniform.

Thus we see that $\brho(\alpha(n,b))$ typically grows slowly, like $\sqrt{n}$, 
except near $b=(-1+\sqrt{5})/4\approx0.3$, 
where our approximate form for the derivative $\cT^{\prime}$ is zero.
We can exploit this behaviour to examine orbits of $\Phi$ in domains $\Xe$ where the twist 
$\kappa(e)$ of $T^e$ is large ($b=0$ ---see figure \ref{fig:resonance}(a)), 
moderate ($b=0.8$ ---see figure \ref{fig:resonance}(b)) 
or almost zero ($b=0.3$ ---see figure \ref{fig:ResonancesCloseUp}).

\begin{figure}[t]
        \centering
        \includegraphics[scale=0.18]{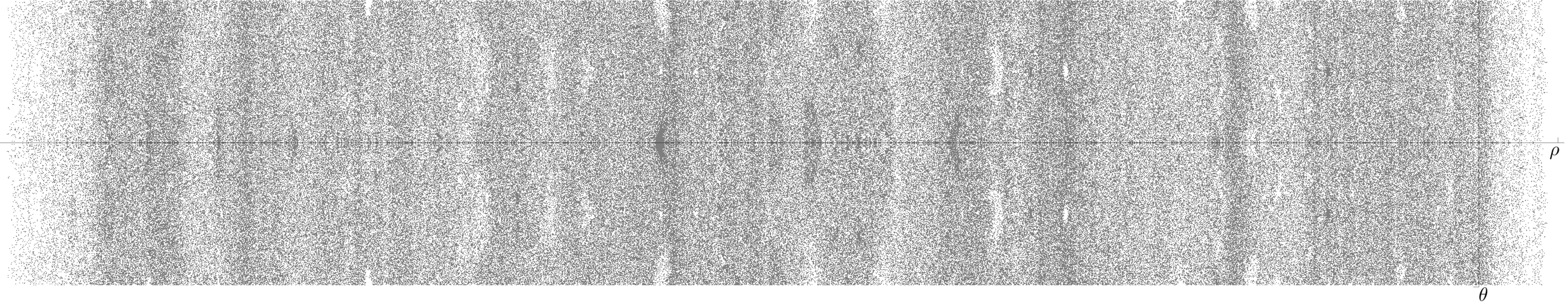} \\
        (a) $e=40000=200^2$ \\
        ~ \\
        \includegraphics[scale=0.18]{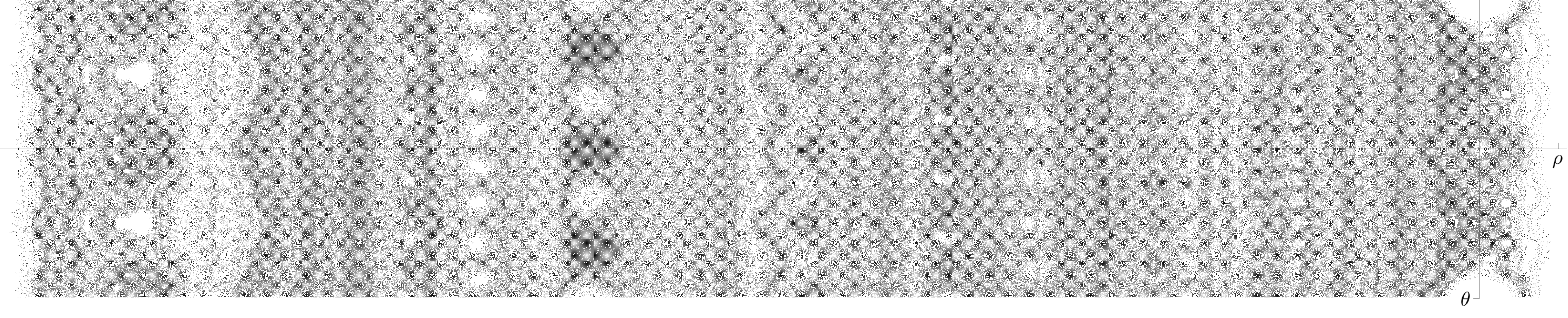} \\
        (b) $e=40309\approx200.8^2$
        \caption{
        Two pixel plots showing a large number of symmetric orbits of the return map $\Phi$ 
        in the cylindrical coordinates $(\theta,\rho)\in\X$.
        The resolution is such that the width of the cylinder (the $\theta$ direction) consists of approximately $280$ lattice sites.
        In both cases, the orbits plotted occupy almost half of the region of phase space pictured.
        The stark contrast between the two plots is caused by the difference in the twist $\kappa(e)$:
        in plot (a) $\kappa(e)\approx 4$, whereas in plot (b) $\kappa(e)\approx -0.1$.
        The values of $\lambda$ used are (a) $\lambda\approx 7\times 10^{-9}$ and (b) $\lambda\approx 4\times 10^{-8}$.
        In figure (b), the primary resonance at the origin is clearly visible, whereas a period $2$
        resonance, which occurs at $\rho=1/2K(e)$, is seen to the left of the plot.}
        \label{fig:resonance}
\end{figure}

\begin{figure}[h]
        \begin{minipage}{7cm}
	  \centering
	  \includegraphics[scale=0.35]{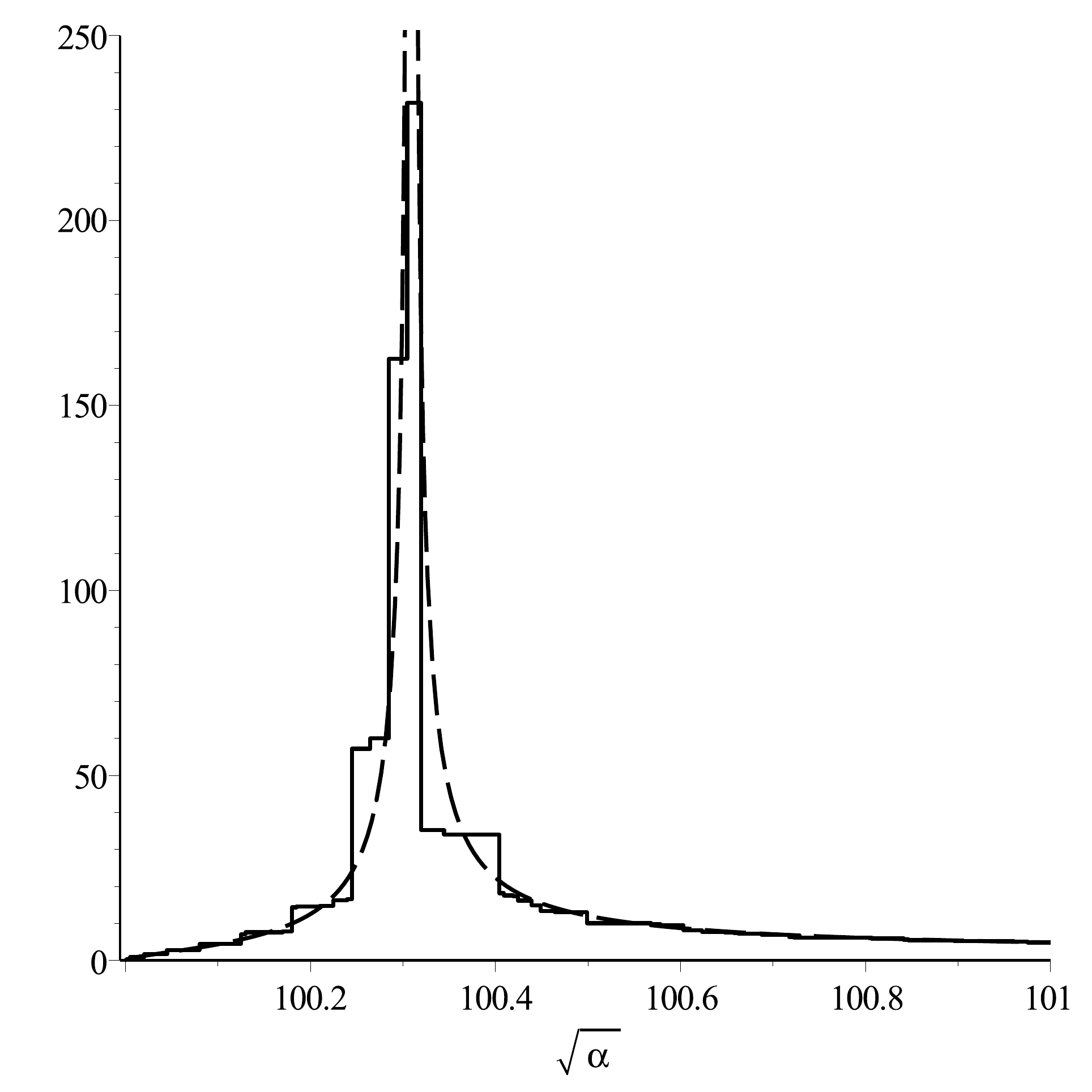} 
	\end{minipage}
        \caption{A plot of $|\brho|$ of equation (\ref{eq:rho_bar}) against $\sqrt{\alpha}$ for 
        $\sqrt{\alpha}\in[100,101)$, i.e., for $\flsq{\alpha}=100$ and $b=\{\sqrt{\alpha}\}\in[0,1)$ (solid line). 
        The dotted line shows the approximate limiting function \eqref{eq:rho_barApprox}.}
        \label{fig:rho_barApprox}
\end{figure}

\begin{table}[t]
        \centering
        \small
        \begin{tabular}{ c c c | c|c | c|c |}
                \cline{4-7}
                 & & & \multicolumn{2}{ |c| }{$\Delta\rho$} & \multicolumn{2}{ |c| }{$\Delta\nu$} \\
                \hline
                \multicolumn{1}{ |c| }{$e$} & \multicolumn{1}{ |c| }{$\vk$} & \multicolumn{1}{ |c| }{$\brho$} 
			      & Median & Maximum & Median & Maximum \\
                \hline
                \multicolumn{1}{|c|}{$10\,000$} & \multicolumn{1}{|c|}{$100$} & \multicolumn{1}{|c|}{$0.266$}  
			    & $0.18$ & $0.65$ & $0.69$ & $2.5$ \\
                \multicolumn{1}{ |c| }{$40\,000$} & \multicolumn{1}{ |c| }{$200$} & \multicolumn{1}{ |c| }{$0.259$} 
			    & $0.16$ & $0.55$ & $0.63$ & $2.1$ \\
                \multicolumn{1}{ |c| }{$160\,000$} & \multicolumn{1}{ |c| }{$400$} & \multicolumn{1}{ |c| }{$0.257$}  
			    & $0.14$ & $0.48$ & $0.54$ & $1.9$\\
                \multicolumn{1}{ |c| }{$640\,000$} & \multicolumn{1}{ |c| }{$800$} & \multicolumn{1}{ |c| }{$0.255$}
			    & $0.13$ & $0.49$ & $0.51$ & $1.9$ \\
                \hline
        \end{tabular}\\[0.2cm]
        (a) $b=0$ \\[0.5cm]

        \begin{tabular}{ c c c  | c|c | c|c |}
                \cline{4-7}
                 & & & \multicolumn{2}{ |c| }{$\Delta\rho$} & \multicolumn{2}{ |c| }{$\Delta\nu$} \\
                \hline
                \multicolumn{1}{ |c| }{$e$} & \multicolumn{1}{ |c| }{$\vk$} & \multicolumn{1}{ |c| }{$\brho$} 
			      & Median & Maximum & Median & Maximum \\
                \hline
                \multicolumn{1}{ |c| }{$10\,057$} & \multicolumn{1}{|c|}{$100$} & \multicolumn{1}{ |c| }{$163$}
			    & $0.14$ & $2.9$ & $8.7\times 10^{-4}$ & $1.8\times 10^{-2}$ \\
                \multicolumn{1}{ |c| }{$40\,113$} & \multicolumn{1}{|c|}{$200$} & \multicolumn{1}{ |c| }{$106$}
			    & $0.25$ & $2.0$ & $2.3\times 10^{-3}$ & $1.9\times 10^{-2}$ \\
                \multicolumn{1}{ |c| }{$160\,234$} & \multicolumn{1}{|c|}{$400$} & \multicolumn{1}{ |c| }{$4105$}
			    & $4.8$ & $8.5$ & $1.2\times 10^{-3}$ & $2.1\times 10^{-3}$ \\
                \hline
        \end{tabular}\\[0.2cm]
        (b) $b=0.3$
        \caption{A table showing the values of $\brho$ for various values of $e$, 
        and the typical range $\Delta\rho$ and $\Delta\nu$ of $\rho(z)$ and $\nu(z)$, respectively, along orbits. 
        The distribution of the range was calculated according to the fraction of points sampled whose orbit has the given range.}
        \label{table:nu_range}
\end{table}

What we see in the case of vanishing twist ($b\neq 0$) is a sea of 
resonance-like structures, resembling the island chains in a 
near-integrable Hamiltonian system.
The global periodicity of the twist map recedes to infinity, and the 
phase portrait is determined by the local rotation number.
Although orbits may wander over a significant range in the $\rho$-direction,
the variation of the rotation number $\nu(z)$ along orbits is vanishingly small
(see table \ref{table:nu_range}(b)).

In the $b=0$ case the global structure remains, and phase portraits appear featureless.
The typical variation in the rotation number along orbits is $1/2$ (see table \ref{table:nu_range}(a)), 
leading to orbits which typically do not cluster in the $\theta$-direction.
In this case it makes sense to consider the statistical properties of orbits of $\Phi$,
and in the next section we consider their period distribution function.


\subsection{The period distribution function} \label{sec:pdf}

For $e=n^2$, i.e., $b=0$, we wish to investigate whether the dynamics of $\Phi$ are sufficiently 
disordered that its period statistics are those of random reversible dynamics.
The following result provides the relevant probabilistic model:

\begin{theorem} \cite[Theorem A]{RobertsVivaldi09} \label{thm:GammaDistribution}
Let $(G,H)$ be a pair of random involutions of a set $\Omega$ with $N$ points, and let
 \begin{equation*} 
  g=\#\Fix{G} \hskip 40pt h=\#\Fix{H} \hskip 40pt \gamma = \frac{2N}{g+h}.
 \end{equation*}
 Let $\cD_N(x)$ be the expectation value of the fraction of $\Omega$ occupied by periodic orbits 
 of $H\circ G$ with period less than $\gamma x$, computed with respect to the uniform probability. 
 If, with increasing $N$, $g$ and $h$ satisfy the conditions
\begin{equation} \label{eq:g_h_conds}
 \lim_{N\rightarrow\infty} g(N) + h(N) = \infty \hskip 40pt \lim_{N\rightarrow\infty} \frac{g(N)+h(N)}{N} = 0,
\end{equation}
 then for all $x\geqslant 0$, we have the limit 
 $$ 
\cD_N(x) \rightarrow \cR(x)= 1 - e^{-x}(1+x).
$$
Moreover, almost all points in $\Omega$ belong to symmetric periodic orbits.
\end{theorem}

The lattice periodicity of $\Phi$ \cite[section 5]{ReeveBlackVivaldi} provides a finite set 
---a discrete torus--- over which to define the period distribution function $\cD^e$ of $\Phi$.
However, as $e\to\infty$, the size of the space grows super-exponentially, making
an exhaustive computation infeasible \cite{ReeveBlack}.
However, all is not lost: if $e=n^2$ is a square, then the phase portrait 
is uniform, and it makes sense to sample the space over the much smaller characteristic 
length $\brho$ of the unperturbed twist map $T^e$.
Furthermore, numerical experiments show that the $\rho$-excursion of a typical periodic orbit 
is smaller than $\brho$. These circumstances will allow us to estimate the global period 
distribution function from local data.

We describe an extensive numerical experiment in which we calculate the period distribution 
function of $\Phi$ over suitably chosen invariant sets, along the subsequence of perfect 
squares $e=n^2$. 
The results suggests that, as $e\to\infty$, the distribution of periods approaches the 
distribution $\cR(x)$ of theorem \ref{thm:GammaDistribution}, and thus is consistent with 
random reversible dynamics. As in the case of the random reversible map, we observe that 
symmetric periodic orbits dominate.

\begin{figure}[h!]
        \centering
        \includegraphics[scale=1.1]{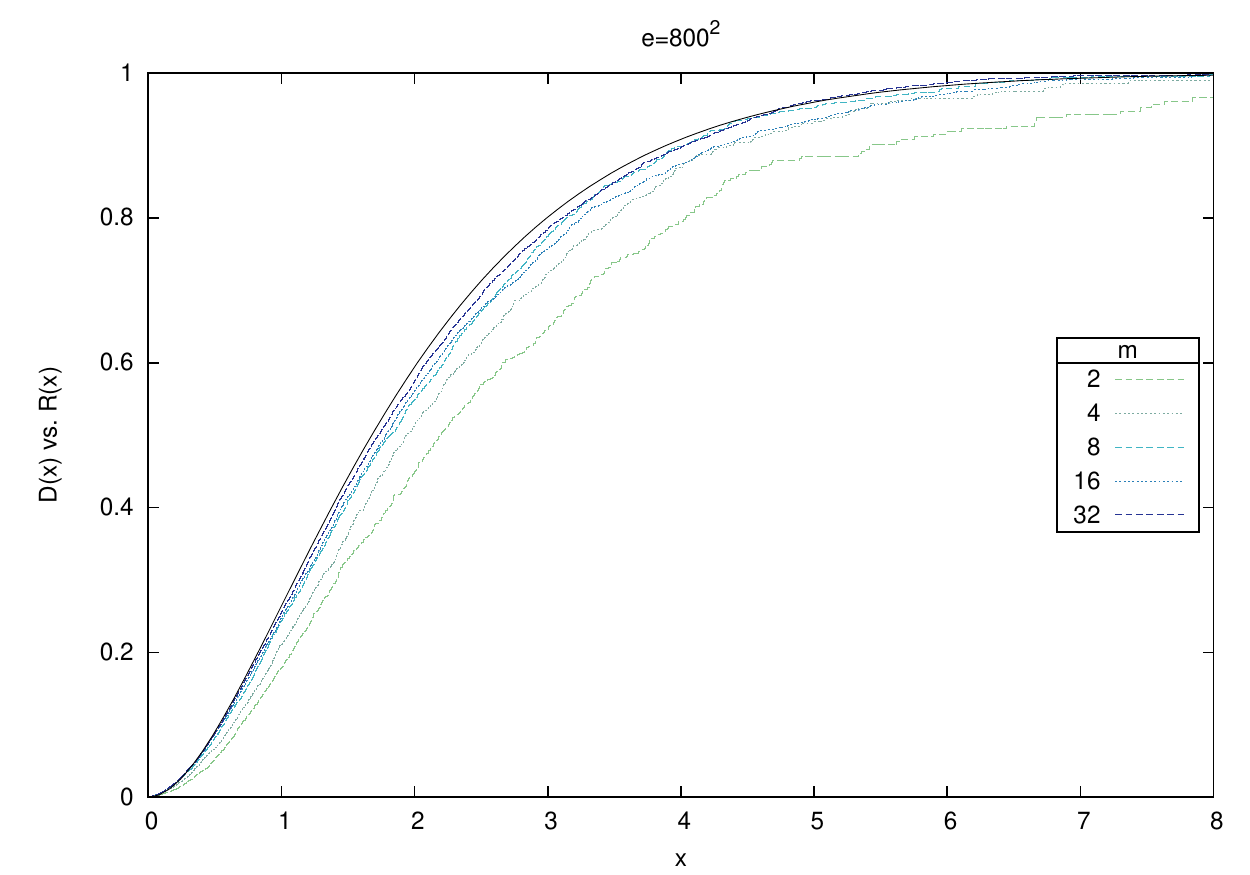} \\
        \caption{A number of the distributions $\cD(x)$ calculated for $n=800$ (green-blue dashed lines). 
        The number $m$ refers to the multiple of the characteristic length scale $\brho$ sampled. 
        The solid black line is the limiting distribution $\cR(x)$.}
        \label{fig:Distribution}
\end{figure}

\begin{table}[h!]
        \centering
        \begin{tabular}{ |c|c|c|c|c|}
                \hline
                $\vk$ & $m$ & Sample size & $\int (\cR-\cD) dx$ & Approx. time \\
                \hline
                $100$ & $32$ & $350\,000$ & $0.03$ & 15sec \\
                $200$ & $32$ & $1\,360\,000$ & $0.04$ & 2min \\
                $400$ & $32$ & $5\,370\,000$ & $0.06$ & 15min\\
                $800$ & $32$ & $21\,200\,000$ & $0.06$ & 2hr \\
                $1600$ & $16$ & $42\,900\,000$ & $0.11$ & 11hr \\
                \hline
        \end{tabular}
        \caption{Some data relating to the calculation of distributions $\cD$ for various values of $n$, 
        and their corresponding maximum values of $m$. 
        For a given value of $m$, nine distinct distributions were calculated: 
        for three different values of $\lambda$, each with three different values of $z_0$ (recall $z_0=(\eta^e)^{-1}(0,0))$.
        Neither the value of $\lambda$ nor $z_0$ was found to have any discernible effect on the distribution.
        The sample size refers to the number of points sampled during the calculation of each distribution, 
        similarly for the approximate computation time. 
        The integral $\int(\cR-\cD)dx$ was calculated over the interval $[0,16]$ (every experimental distribution 
        satisfies $\cD(16)=1$), and was averaged over the nine individual distributions calculated.}
        \label{table:D_table}
\end{table}

\medskip

For $e=n^2\in\cE$, we wish to calculate a sequence of period distribution functions, 
which will serve as approximations to $\cD^e$ as $e\to\infty$.
We calculate these period distribution functions over a sequence of $\Phi$-invariant sets 
which mimic the fundamental domains defined in \eqref{eq:fundamental_domain} ---the natural 
invariant structures of the twist dynamics.

We consider subsets $A=A(e,m,\lambda)$ of $\Xe$ of the form
\begin{equation} \label{eq:A_m}
 A(e,m,\lambda) = \{ z\in\Xe \,: \; \nu(z) \in [-1/2,m-1/2] \} \hskip 40pt m\in\N. 
\end{equation}
For sufficiently small $\lambda$, the counterpart of $A$ on the cylinder
covers $m$ copies of the fundamental domain of the twist dynamics, so that as $e\to\infty$:
 $$ \# A \sim 2m(2\vI+1)^2|\brho| \to \infty $$
(cf.~equation \eqref{eq:pts_per_fundamental_domain}).

The set $A$ is not invariant under the perturbed dynamics $\Phi$. 
Hence we define $\bA$ to be the smallest invariant set which contains $A$:
\begin{equation*} 
 \bA(e,m,\lambda) = \bigcup_{k\in\Z} \Phi^k(A(e,m,\lambda)).
\end{equation*}
The finiteness of the set $\bA$ relies on the conjectured periodicity of all 
orbits (section \ref{sec:Introduction}).

We collect our findings as a series of observations.
First, we observe that the overspill from $A$ under the map $\Phi$, 
i.e., the set $\bA\setminus A$, 
is small relative to $A$ as $m\to\infty$ (see figure \ref{fig:NoOfPts}(a)).
On the basis of this observation, in what follows, it is assumed that there is a critical parameter value $\lambda_c(e,m)$ 
such that $\bA\subset\Xe$ for all $\lambda<\lambda_c$.

\begin{observation}
Let $e=n^2\in\cE$. Then for a sequence $(\lambda_m)$ of parameter values satisfying 
$\lambda_m<\lambda_c(n^2,m)$, we have:
$$ 
\frac{\#\bA(n^2,m,\lambda_m)}{\# A(n^2,m,\lambda_m)} \to 1\qquad m\to\infty.
$$
\end{observation}

Then we measure the period distribution function $\cD=\cD(n^2,m,\lambda)$ of $\Phi$ over $\bA$:
\begin{displaymath}
 \cD(x) = \frac{\# \{ z\in \bA \,: \; \tau(z)\leqslant \gamma x \}}{\# \bA},
\end{displaymath}
where the scaling constant $\gamma=\gamma(n^2,m,\lambda)$ is given by
\begin{equation} \label{eq:Phigamma_approx}
 \gamma = \frac{2\#\bA}{g+h} \hskip 20pt g=\#\left(\Fix{G^e}\cap\bA\right) \hskip 20pt h=\#\left(\Fix{(\Phi\circ G^e)}\cap\bA\right).
\end{equation}

Assuming that the torus structure used to define $\cD^e$ does not affect the period
of the orbits, then, as $m\to\infty$, i.e., as we sample an increasing number of
characteristic lengths $\brho$, the empirical period distribution approaches $\cD^e$.
As $n$ increases, we need to ensure that $m$ is large enough to achieve numerical convergence.
In practice, small values of $m$ were sufficient (see table \ref{table:D_table}).

\medskip

Since $\Fix{G^e}$ is the pair of lines $x=y$ and $x-y=-\lambda(2\vI+1)$, 
the corresponding set on the cylinder is given by (cf.~\eqref{eq:rho_theta_lattice})
 $$ \eta^e(\Fix{G^e}) = \left\{ \frac{1}{2(2\vI+1)} \, (i,i+2j)\,: \; i\in\{-(2\vI+1),0\}, \; j\in\Z \right\}. $$
Intersecting this with $A$ (see equation (\ref{eq:A_m})) restricts the index $j$ according to
$$ 
\frac{i+2j}{2|\brho|(2\vI+1)} \in \left[ -\frac{1}{2}, m-\frac{1}{2} \right).
$$
Thus, equating $A$ with $\bA$ in the limit, we have
$$ 
g \sim \#\left(\Fix{G^e}\cap A\right) \sim 2m|\brho|(2\vI+1) \to \infty 
$$
as $m,n\to\infty$. 
The fixed space $\Fix{(\Phi\circ G^e)}$ is the lattice equivalent of the line $\Fix{H^e}$ 
of equation \eqref{eq:Fix(cH)}. 
We have the following experimental observation for the size of $h$ (see figure \ref{fig:NoOfPts}(b)).

\begin{observation} \label{obs:g_h}
Let $e=n^2\in\cE$. Then for a sequence $(\lambda_{n,m})$ of parameter values satisfying 
$\lambda_{n,m}<\lambda_c(n^2,m)$, we have:
$$ 
h \sim \frac{g}{\sqrt{2}} \hskip 40pt m,n\to\infty.
$$
\end{observation}

\begin{figure}[t]
        \centering
        \begin{minipage}{7cm}
	  \centering
	  \includegraphics[scale=0.55]{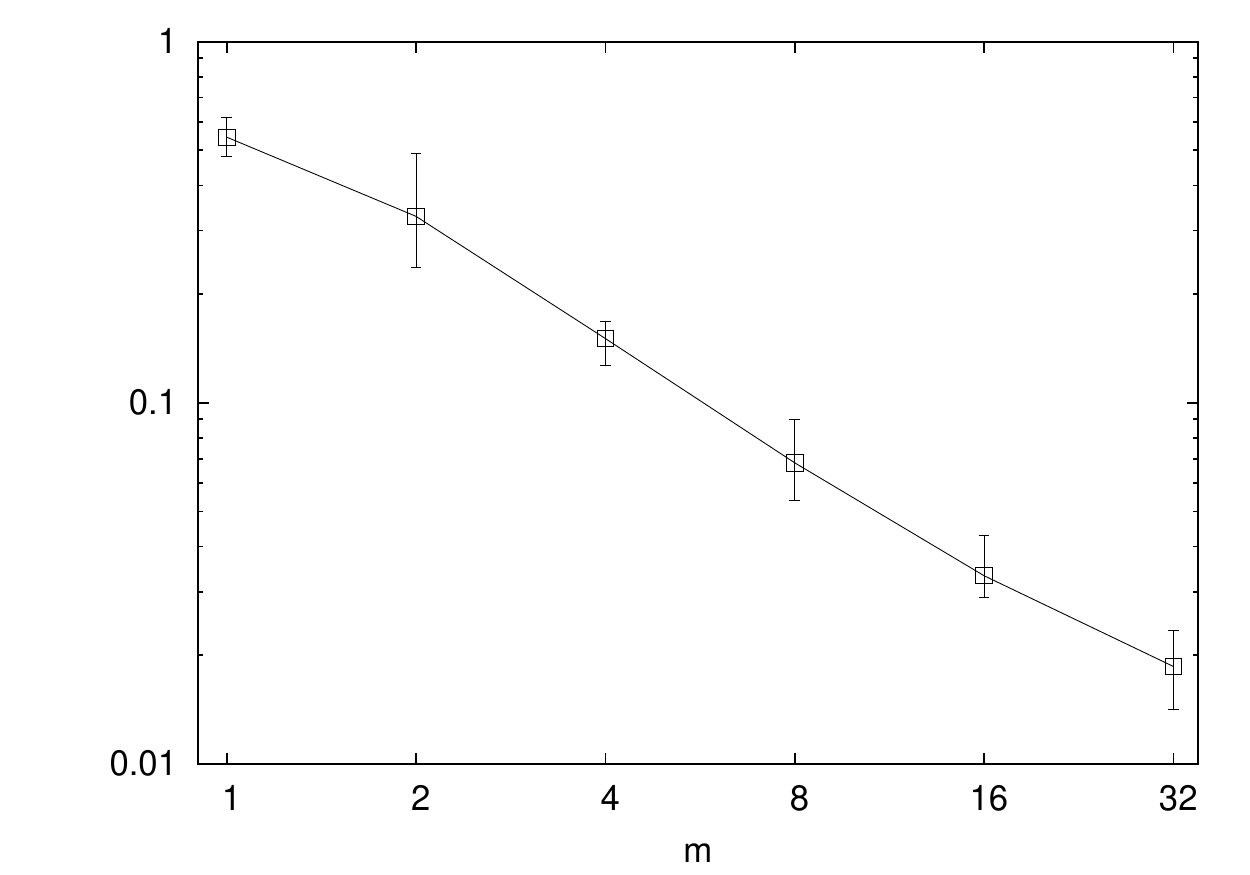} \\
	  (a) $\# \bA/\# A-1$ \\
	\end{minipage}
        \quad
        \begin{minipage}{7cm}
          \centering
	  \includegraphics[scale=0.55]{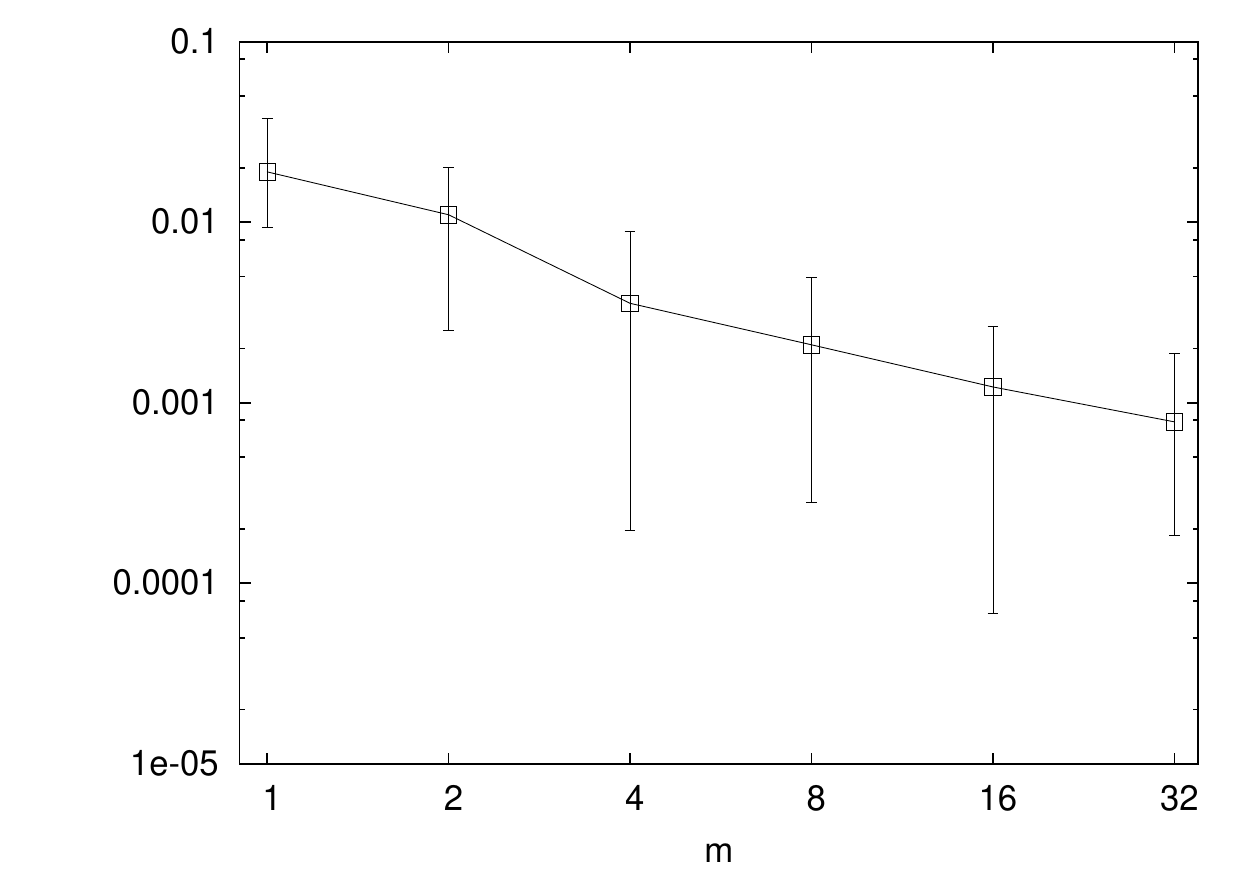} \\
	  (b) $h/g - 1/\sqrt{2}$ \\
        \end{minipage}
        \caption{The convergence (a) of the ratio $\# \bA/\# A$ to $1$ and (b) of the ratio $h/g$ to $1/\sqrt{2}$ as $m$ becomes large, for $\vk=800$. The line shows the average value of the relevant ratio among all experiments performed: the error bars indicate its minimum and maximum value. All axes are displayed with a logarithmic scale.}
        \label{fig:NoOfPts}
\end{figure}


From this observation, it follows that 
$$ 
\frac{g+h}{\#\bA} \sim \frac{(2+\sqrt{2})}{2(2\vI+1)} \to 0 
$$
as $m,n\to\infty$, and hence that the quantities $g$ and $h$ satisfy the conditions 
\eqref{eq:g_h_conds} of theorem \ref{thm:GammaDistribution}.
Indeed, we observe that the universal distribution $\cR(x)$ 
is the limiting distribution for $\cD$ in the limits $m,n\to\infty$ 
(see figure \ref{fig:Distribution}).

\begin{observation} \label{obs:De}
With $e=n^2$ and $(\lambda_{n,m})$ as in observation \ref{obs:g_h}, we have:
$$ 
\cD(n^2,m,\lambda_{n,m}) \to \cR 
\qquad m,n\to\infty,
$$
where $\cR$ is the universal distribution of equation \eqref{eq:R(x)}.
\end{observation}

Finally we note that, as in theorem \ref{thm:GammaDistribution}, 
the symmetric orbits of $\Phi$ have full density (see figure \ref{fig:symm_points}).

\begin{observation} \label{obs:symm_points}
Let $S=S(e,m,\lambda)$ be the set of points in $\bA$ whose orbit under $\Phi$ is symmetric:
$$ S = \{ z\in \bA \,: \; \cO(z)=G^e(\cO(z)) \}. $$
Furthermore, let $e=n^2$ and $(\lambda_{n,m})$ be as in observation \ref{obs:g_h}.
Then, asymptotically, $S$ has full density in $\bA(m)$:
$$ 
\frac{\#S(n^2,m,\lambda_{n,m})}{\#\bA(n^2,m,\lambda_{n,m})}\to 1
\qquad m,n\to\infty.
$$
\end{observation}

\begin{figure}[!h]
        \centering
        \includegraphics[scale=0.55]{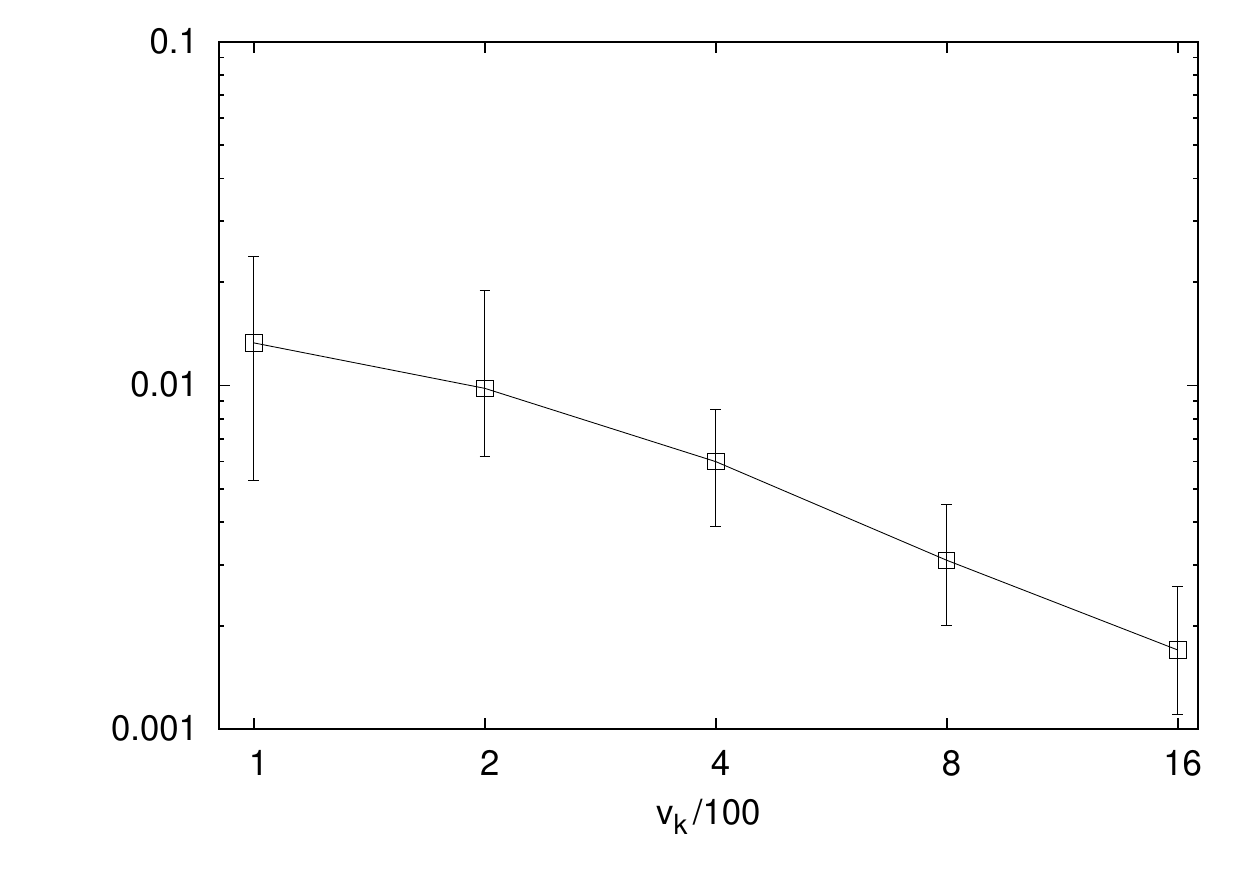} 
        \caption{The quantity $1-\#S/\#\bA$ as $\vk=n$ becomes large. 
        The line shows the average value of the relevant ratio over all experiments performed 
        (including over $m=1,2,4,8,16,32$---unlike the distribution $\cD$ of figure \ref{fig:Distribution}, 
this ratio does not vary significantly with $m$): the error bars indicate its minimum and maximum value.
        The axes are displayed with a logarithmic scale.}
        \label{fig:symm_points}
\end{figure}


\section*{Acknowledgements} We thank John A. G. Roberts for useful discussion.
HRB thanks the School of Mathematics and Statistics at the University of New
South Wales, Sydney, for their hospitality while part of this work was carried out. 

\bibliographystyle{alpha}
\bibliography{Bibliography}

\end{document}